\documentclass[11pt,twoside]{amsart}
\usepackage{amssymb,amscd,color}
\usepackage{amsmath}
\usepackage{tikz}
\usetikzlibrary{matrix}
\usepackage{tikz-cd}
\usetikzlibrary{calc, intersections, through}
\usepgflibrary{arrows}
\usetikzlibrary{positioning}
\usepackage{fancyhdr}
\usepackage{array}
\usepackage[all]{xy}
\usepackage{pdflscape}
\usepackage{afterpage}
\usepackage{capt-of}
\usepackage{booktabs}
\usepackage{enumitem}
\usepackage{lipsum}
\usepackage{graphicx,caption, subcaption}
\usepackage{bbm}
\usepackage[symbol]{footmisc}
\usepackage{xcolor}
\usepackage[square,numbers]{natbib}
\usetikzlibrary{arrows.meta}
\usepackage{comment}
\usepackage{hyperref}

\numberwithin{equation}{section}

\tikzstyle{wbullet}=[circle, draw=black, fill=white, thick, inner sep=0pt, minimum size=1.5mm]
\tikzstyle{wbullet}=[circle, draw=black, fill=white, thick, inner sep=0pt, minimum size=1.5mm]
\tikzstyle{bbullet}=[circle, draw=black, fill=black, inner sep=0pt, minimum size=1.5mm]
\tikzstyle{cross}=[circle, draw=black, fill=white, inner sep=0pt, minimum size=1.5mm]

\usepackage[
]{hyperref}

\usepackage{todonotes}
\theoremstyle{plain}
\newtheorem{theorem}{Theorem}[section]
\newtheorem{lemma}[theorem]{Lemma}
\newtheorem{corollary}[theorem]{Corollary}
\newtheorem{proposition}[theorem]{Proposition}

\theoremstyle{definition}
\newtheorem{definition}[theorem]{Definition}

\newtheorem{remark}[theorem]{Remark}
\newtheorem{claim}[theorem]{Claim}
\newtheorem{question}[theorem]{Question}
\newtheorem{notation}[theorem]{Notation}

\newcommand{\bA}{\mathbb{A}}

\newcommand{\bC}{\mathbb{C}}

\newcommand{\bP}{\mathbb{P}}
\newcommand{\bQ}{\mathbb{Q}}
\newcommand{\bR}{\mathbb{R}}
\newcommand{\bZ}{\mathbb{Z}}

\newcommand{\cC}{\mathcal{C}}
\newcommand{\cD}{\mathcal{D}}

\newcommand{\cO}{\mathcal{O}}

\newcommand{\cS}{\mathcal{S}}

\newcommand{\pet}{\mathrm{pet}}

\newcommand{\lct}{\mathrm{lct}}

\newcommand{\Supp}{\mathrm{Supp}}
\newcommand{\vol}{\mathrm{vol}}
\newcommand{\ord}{\mathrm{ord}}
\newcommand{\Exc}{\mathrm{Exc}}

\newcommand{\mult}{\mathrm{mult}}

\newcommand{\divisor}{\mathrm{div}}

\makeindex

\title[Minimal volume of normal stable surface pairs]
{The Minimal Volume of Normal Stable Surface Pairs with Reduced Boundary Containing a Zero-Dimensional Non-klt Center}
\author{Jihao Liu, Weili Shao}
\address{Department of Mathematics, Peking University, No. 5 Yiheyuan Road, Haidian District, Beijing
100871, China}
\address{Beijing International Center for Mathematical Research, Peking University, No. 5 Yiheyuan
Road, Haidian District, Beijing 100871, China}
\email{liujihao@math.pku.edu.cn}
\address{School of  Mathematical Sciences, Xiamen University, Xiamen, Fujian 361005, P.~R.~China}
\email{wlshaomath@stu.xmu.edu.cn}


\begin{document}

\begin{abstract}
    We show that the minimum volume among normal stable surface pairs with nonzero reduced boundary containing a zero-dimensional non-klt center is $\frac{1}{42}$, and show that the surface pair attaining this minimum is unique up to isomorphism. As an application, we show that $\frac{1}{42}$ is an infinite accumulation point of the volume set of normal stable surfaces.
\end{abstract}

\maketitle

% The introduction will be added after the mathematical sections have
% reached their final form.

\tableofcontents
\enlargethispage{3pt}

\section{Introduction}

We work over $\bC$. A normal stable surface pair $(X,B)$ is a projective log canonical surface pair for which $K_X+B$ is an ample $\bR$-Cartier divisor. When $B=0$, we call $X$ a normal stable surface. Normalizing non-normal stable surfaces in KSBA compactifications yields normal stable surface pairs with reduced conductor boundary \cite[Theorem~5.13]{Kollar13book}.

Determining minimal volumes and classifying the surfaces and pairs attaining them can provide a natural starting point for the explicit study of the associated KSBA moduli spaces; see, for example, \cite{LiuLiu26minimal,alexeevliuschutt2025modulispacestablesurfaces}. Minimal volumes are known for several classes of normal stable surfaces and surface pairs \cite{alexeevliu19picardone,AlexeevLiu19opensurfaces,LiuShokurov2023optimal,Liu26positivegenus,liuliu26minimalvolumestablesurfaces,LiuLiu26minimal}. In particular, Alexeev and W.~Liu \cite{AlexeevLiu19opensurfaces} constructed a normal stable surface pair with nonzero reduced boundary and volume $\frac{1}{462}$, and the first author and Shokurov \cite[Theorem~1.4]{LiuShokurov2023optimal} later proved that $\frac{1}{462}$ is the minimal volume in this class. For normal stable surfaces, however, the minimum remains unknown. The smallest known volume is $\frac{1}{48983}$, realized by a klt surface first constructed by Alexeev and W. Liu \cite{AlexeevLiu19opensurfaces}; see also \cite{Totaro24minimalvolume}. This value is conjectured to be the minimum. 

In contrast, the smallest accumulation point of the volumes of normal stable surfaces is known: the first author and W.~Liu \cite[Theorem~1.1 and Corollary~1.2]{LiuLiu26minimal} proved that it is $\frac{1}{825}$, which is also the minimal volume among such surfaces with non-empty non-klt locus.

To discuss higher order accumulation points, we introduce the following notation. Given a set $\cC\subset [0,1]$, we set
\begin{itemize}
    \item $\cS(2,\cC)$: the set of projective lc surface pairs $(X, B)$ with $B\in \cC$ and $K_X+B$ ample;
    \item $\mathrm{Vol}(2,\cC)$: the set of volumes $\mathrm{vol}(X, K_X+B)$ for $(X, B)\in \cS(2,\cC)$; 
    \item $\mathrm{Vol}^{(0)}(2,\cC):=\mathrm{Vol}(2,\cC)$, and let $\mathrm{Vol}^{(k+1)}(2,\cC)$ be the derived set of $\mathrm{Vol}^{(k)}(2,\cC)$ for $k\geq0$.
\end{itemize}

If $v\in \mathrm{Vol}^{(k)}(2,\cC)$ for all $k\geq 1$, then $v$ is called an \emph{infinite accumulation point} of $\mathrm{Vol}(2,\cC)$ and we denote by $\mathrm{Vol}^{(\infty)}(2,\cC)$ the set of all infinite accumulation points of $\mathrm{Vol}(2,\cC)$ (see \cite[Definition 1.1]{Shao26p}). The second author proved that $\mathrm{Vol}^{(\infty)}(2,\{0\})\neq \emptyset$, and hence $\mathrm{Vol}^{(\infty)}(2,\cC)\neq \emptyset$ for every set $\cC\subseteq [0,1]$. This raises the following problem in dimension two.

\begin{question}[{\cite[Question 4.3]{Shao26p}}]\label{ques}
    What is the smallest infinite accumulation point $m_2$ of $\mathrm{Vol}(2,\{0\})$? 
\end{question}

The existence result is based on the following geometric construction.

\begin{theorem}[{\cite[Theorem 1.4]{Shao26p}}]\label{thm:shao26}
    Let $(X, B)$ be a $\bQ$-factorial projective lc surface with $K_X+B$ ample. Suppose that there exists a closed point $x$ on an irreducible component $B_0$ of $B^{=1}$ such that $x$ is a non-klt center of $(X, B)$. Then
    \begin{align*}
        \mathrm{vol}(X, K_X+B)\in \mathrm{Vol}^{(\infty)}(2, \cC_B),
    \end{align*} where $\cC_B$ is the coefficient set of $B$.
\end{theorem}

Motivated by this construction, we determine the minimal volume among projective lc surface pairs $(X,B)$ such that $K_X+B$ is ample, $B$ is a nonzero reduced divisor and contains a zero-dimensional non-klt center, and classify the pairs attaining it. We further show that this minimum is an infinite accumulation point of $\mathrm{Vol}(2,\{0\})$, thereby obtaining an upper bound for $m_2$.

More precisely, let $\cS'(2,\{0,1\})$ denote the subset of $\cS(2,\{0,1\})$ consisting of pairs $(X,B)$ such that $B$ contains a zero-dimensional non-klt center of $(X,B)$.
\begin{theorem}\label{thm:main-thm}
    If $(X, B)\in \cS'(2, \{0,1\})$, then
    \begin{align*}
        (K_X+B)^2\geq \frac{1}{42}.
    \end{align*}
    Moreover, the surface pair that achieves this minimum is uniquely determined up to isomorphism, and $\frac{1}{42}\in \mathrm{Vol}^{(\infty)}(2,\{0\})$. In particular, $m_2\leq \frac{1}{42}$.
\end{theorem}

This theorem is closely related to the determination of $m_2$. Indeed, an affirmative answer to \cite[Question~4.2]{Shao26p}, asking whether every infinite accumulation point of $\mathrm{Vol}(2,\{0\})$ is realized as the volume of a pair in $\cS'(2,\{0,1\})$, would imply $m_2=\frac{1}{42}$ by Theorem~\ref{thm:main-thm}.

\medskip
\noindent\textit{Sketch of the proof.}
In Section~2, we collect the notation and basic properties of surface pairs, complements, and toroidal pairs. 

In Section~3, we establish the lower bound $\frac{1}{42}$. Let $(X,B)\in \cS'(2,\{0,1\})$ and choose an irreducible component $B_0$ of $B$ through a zero-dimensional non-klt center $x$. Set $b:=\operatorname{pet}(X,0;B)$. Adjunction formula and the structure of non-klt centers give a uniform estimate $(K_X+B)\cdot B_0\geq \frac{1}{6}$. If $b\leq 6/7$, then the definition of the pseudo-effective threshold yields
\begin{align*}
    (K_X+B)^2
    \geq (1-b)(K_X+B)\cdot B_0
    \geq \frac{1}{42}.
\end{align*}
If $b>6/7$, we obtain the stronger estimate $(K_X+B)\cdot B_0\geq \frac{1}{2}$. Combining this with the known gaps for pseudo-effective thresholds (\cite[Theorem 7.7]{LiuShokurov2023optimal}, \cite[Theorem 2.6]{LiuLiu26minimal}), we obtain $(K_X+B)^2\geq 1/22$. Thus $\frac1{42}$ is a lower bound for the volumes of pairs in $\cS'(2,\{0,1\})$.

In Section~4, we show that the lower bound is sharp. Starting from the exceptional pair of type $(I_2^2)$ on $\mathbb P(1,2,3)$, we perform a suitable toroidal extraction over a boundary crossing. The resulting log pair $(Y,B_Y+E)$ has a big and nef log canonical divisor with self-intersection $\frac1{42}$. Passing to the ample model of \(K_Y+B_Y+E\) and taking the induced boundary gives a pair in $\cS'(2,\{0,1\})$ attaining the volume $\frac1{42}$. This proves that the lower bound in Section 3 is optimal.

We next prove the uniqueness (Theorem \ref{thm:uniqueness}). Let $(X,B)$ be a pair attaining the minimal volume. In Section 5, we establish the basic structure of such an pair: $(X, \frac{6}{7}B)$ is an exceptional log Calabi-Yau surface pair with $\delta(X, \frac{6}{7}B)=2$ (Theorem \ref{thm:equality-properties}).

In Section 6, we first show that after extracting an lc place (Lemma \ref{lem:boundary-center-extraction}), the resulting pair $(Y, B_Y+E)$ admits a birational morphism $h\colon (Y, B_Y+E)\to (S,C_1+C_2)$, where $(S, \frac{6}{7}(C_1+C_2))$ is of type $(A_2^6)$ or $(I_2^2)$ in Shokurov's classification of rank one exceptional surface pairs (Lemma \ref{lem:rank-one-reduction}).

Using the numerical constraints on the volumes and intersection numbers, we further show that the exceptional locus of $h$ has only two possible configurations and characterize $h$ explicitly (Lemma \ref{lem:exhaustion}, Proposition \ref{prop:reduction-to-two-types}).

Finally, for each possibility, we construct sections of degrees $2,6,7,$ and $9$ on $S$ whose lifts to $Y$ define a birational map onto a hypersurface $X_{18}\subseteq \bP(2,6,7,9)$. A volume comparison argument (Proposition \ref{prop:ample-model-criterion}) shows that its image $X_{18}$ is the ample model of $K_Y+B_Y+E$ (Proposition \ref{prop:ample-model-two-types}). By construction of $(Y, B_Y+E)$, $X$ is the ample model of $K_Y+B_Y+E$. Hence $X\simeq X_{18}$. Finally, identifying the pushforward of $B_Y$ with $B_{18}\subseteq X_{18}$ (Proposition \ref{prop:ample-model-boundary}) then proves the uniqueness: $(X, B)\simeq (X_{18}, B_{18})$.

In Section 7, we combine the construction of infinite accumulation points (\cite[Proof of Proposition 3.1]{Shao26p}) with the criterion for accumulation points (\cite[Theorem 1.1]{AlexeevLiu19accpoint}) to show that $\frac{1}{42}\in\operatorname{Vol}^{(\infty)}(2,\{0\})$.

\medskip
\noindent{\bf Acknowledgments.} 
This project began when the second author visited the first author at Peking University in September 2025. The second author is grateful for the warm hospitality during this visit. We would like to thank members of the Danus team (namely Guoxiong Gao, Zeming Sun, Bin Wu, Shurui Liu, Jiedong Jiang, Haocheng Ju, Leheng Chen, Ronnie Cheng, Xiping Zhang, and Bin Dong) and the Rethlas team (namely Haocheng Ju, Jiedong Jiang, Shurui Liu, Guoxiong Gao, Yuefeng Wang, Zeming Sun, Bin Wu, Liang Xiao, and Bin Dong) for their contributions to the development of Danus and Rethlas. The first author would like to thank Ruochuan Liu and Gang Tian for constant support and encouragement. The second author would like to thank his advisor Wenfei Liu for helpful discussion.

The first author was partially supported by the National Key R\&D Program of China (\#2024YFA1014400). The second author was partially supported by the National Natural Science Foundation of China (No. 12571046).

\medskip
\noindent\textbf{AI disclosure.}
The authors used the Danus system and ChatGPT 5.6 to assist in developing parts of the lower bound and uniqueness arguments, checking proofs, and polishing the exposition. All AI-assisted material was independently verified and revised by the authors, who take full responsibility for the contents of the paper.

\section{Preliminaries}
\label{sec:preliminaries}

We use the standard terminology of the minimal model program as in \cite{Kollar92Utah, KM98, BCHM}. Set
\begin{align*}
    \Phi_{\mathrm{sm}}:=\left\{1-\frac{1}{m}\ \middle|\ m\in\bZ_{\geq2}\right\}\cup \{1\}, \qquad \Phi_\mathrm{m}:=\Phi_{\mathrm{sm}}\cup [\frac{6}{7},1].
\end{align*}

\subsection{Surface pairs and numerical invariants}

A pair $(X,B)$ consists of a normal projective variety $X$ and an effective $\bR$-divisor $B$ such that $K_X+B$ is $\bR$-Cartier. A divisor \emph{over} $X$ is a prime divisor $E$ on some normal variety $Y$ equipped with a birational morphism $\pi\colon Y\to X$. The closure of the image $\pi(E)$ is called the \emph{centre} of $E$ on $X$ and denoted by $c_X(E)$. Two divisors over $X$ are usually identified if they define the same valuation on the function field of $X$.

Let $\pi\colon Y\to X$ be a proper birational morphism from a normal variety, and let $E$ be a prime divisor on $Y$. For an effective $\bR$-Cartier divisor $D$ on $X$, set $\ord_E(D):=\mult_E(\pi^*D)$. Here $\mult_E$ denotes the coefficient of $E$ in the corresponding divisor.

Let $D$ be a $\bQ$-Cartier Weil divisor on $X$. A nonzero section $s\in H^0(X,\cO_X(D))$ corresponds to a rational function $h\in K(X)^*$ satisfying
\begin{align*}
    \divisor_X(h)+D\geq 0.
\end{align*}
The zero divisor of $s$ is defined by
\begin{align*}
    D_s:=(s=0):=\divisor_X(h)+D.
\end{align*}
In particular, $D_s$ is an effective $\bQ$-Cartier divisor. We define $\ord_E(s):=\ord_E(D_s)$.

The \emph{log discrepancy} of $E$ with respect to $(X, B)$ is
\begin{align*}
    a(E,X,B):=1+\mult_E\bigl(K_Y-f^*(K_X+B)\bigr),
\end{align*}
where $f\colon Y\to X$ is a model on which $E$ appears. We say that $(X,B)$ is \emph{lc} if $a(E,X,B)\geq0$ for every $E$, and \emph{klt} if $a(E,X,B)>0$ for every $E$. For an lc pair $(X,B)$, an irreducible subvariety $Z\subseteq X$ is called a \emph{non-klt center} if $Z=c_X(E)$ for some prime divisor $E$ over $X$ with $a(E,X,B)=0$.

\begin{lemma}\label{lem:lc -pair-underlying-klt}
Let $(X,B)$ be a lc surface pair. If $x\in\mathrm{Supp}(B)$, then $X$ is klt at $x$.
\end{lemma}

\begin{proof}
    This follows from the classification of lc surface germs; see, for example, \cite[Chapter 4]{KM98}.
\end{proof}

Following \cite[Remark~2.4]{AlexeevLiu19accpoint}, a non-klt center of an lc surface pair $(X,B)$ is called \emph{inaccessible} if it is either a simple elliptic singularity $p\notin\Supp(B)$, or a smooth component $B_j$ of $B$ with coefficient one, contained in the smooth locus of $X$ and disjoint from $\Supp(B-B_j)$. All other non-klt centers are called \emph{accessible}.

For an $\bR$-Cartier divisor $D$ on a projective surface, set
\begin{align*}
    \vol(X,D) :=\limsup_{m\to\infty} \frac{2h^0(X,\lfloor mD\rfloor)}{m^2}.
\end{align*}
If $D$ is big and nef, then $\vol(X,D)=D^2$.

Let $(X, B)$ be a projective lc pair, and let $D\geq 0$ be an $\bR$-Cartier $\bR$-divisor on $X$. We define the lc threshold and the pseudo-effective threshold by
\begin{align*}
    \lct(X, B;D):=&\sup\left\{t\geq0\ \middle|\ (X, B+tD) \text{ is lc }\right\}\\
    \pet(X,B;D):=&\inf \{t\geq0 | (X, B+tD) \text{ is lc },\\
    &K_X+B+tD \text{ is pseudo-effective} \}.
\end{align*}

\begin{lemma}[{\cite{Kuwata99lct-red-curve}, \cite[Corollary~3.3]{Prokhorov02lct-II}}]\label{lem:surface-one-gap}
Let $(X,T)$ be a lc surface pair with $T$ reduced and let $S$ be a nonzero effective $\bQ$-Cartier Weil divisor. Then $\mathrm{lc t}(X,T;S)=1$ or $\mathrm{lc t}(X,T;S)\leq\frac{5}{6}$. 
\end{lemma}

We shall repeatedly use surface adjunction in the following form; see \cite[Definition~2.34 and Theorem~3.36]{Kollar13book}. Let $(X,B)$ be an lc surface pair with $B$ reduced. Let $C$ be an irreducible component of $B$, and let $\nu\colon C^\nu\to C$ be its normalization. Then
\begin{align*}
    \nu^*((K_X+B)|_C)
    =K_{C^\nu}+\mathrm{Diff}_{C^\nu}(B-C).
\end{align*}
The coefficients of the $\mathrm{Diff}_{C^\nu}(B-C)$ belong to $\Phi_{\mathrm{sm}}$.

\begin{lemma}\label{lem:positive-branch-one-sixth}
Let $(X,B)$ be a projective lc surface pair with $B$ a nonzero reduced divisor and $K_X+B$ is nef. Suppose that there exists an irreducible component $C$ of $B$ containing a zero-dimensional non-klt center $x$. If $(K_X+B)\cdot C>0$, then
\begin{align*}
    (K_X+B)\cdot C \geq\frac{1}{6}.
\end{align*}
Moreover, equality holds if and only if
\begin{align*}
    C\simeq\mathbb P^1, \qquad 
    \operatorname{Diff}_{C}(B-C)=x+\frac{1}{2}p_2+\frac{2}{3}p_3
\end{align*}
for three distinct closed points $x,p_2,p_3\in C$.
\end{lemma}

\begin{proof}
By inversion of adjunction, there exists a point $p_\infty\in C^\nu$ lying over $x$ whose coefficient in $\operatorname{Diff}_{C^\nu}(B-C)$ is equal to one. If $g(C^\nu)\geq 1$, then
\begin{align*}
    (K_X+B)\cdot C&=2g(C^\nu)-2+\deg \mathrm{Diff}_{C^\nu}(B-C)\geq 1,
\end{align*}
so the assertion follows.

Otherwise, $C^\nu\simeq\mathbb P^1$. Since the coefficients of $\mathrm{Diff}_{C^\nu}(B-C)$ belong to $\Phi_{\mathrm{sm}}$ and the different has a point of coefficient one, the inequality $(K_X+B)\cdot C>0$ implies that
\begin{align*}
    (K_X+B)\cdot C\geq -2+1+\frac{1}{2}+\frac{2}{3}=\frac{1}{6}.
\end{align*}
Moreover, equality holds if and only if
\begin{align*}
    C^\nu\simeq \bP^1,\qquad
    \operatorname{Diff}_{C^\nu}(B-C)
    =
    p_\infty+\frac{1}{2}p_2+\frac{2}{3}p_3
\end{align*}
for three distinct points $p_\infty,p_2,p_3\in C^\nu$. It remains to show that $C$ is a smooth rational curve in the equality case. If $C$ is singular, log canonicity of $(X, B)$ implies that $C$ has a node. The preimages in $C^\nu$ of the node contribute two points in the $\mathrm{Diff}_{C^\nu}(B-C)$. Contradiction. Thus $C\simeq \bP^1$ and $p_\infty=x$.
\end{proof}

\subsection{Complements}

In this section, we collect the notation and basic facts on complements, we refer the reader to \cite{Shokurov00Complements-on-surfaces,Prokhorov01lectures, ChenHan21complements}.

\begin{definition}
    Let $(X, B)$ be a log pair. A \emph{$\bQ$-complement} of $K_X+B$ is a log divisor $K_X+B^+$ such that $B^+\geq B$, $(X, B^+)$ is lc and $K_X+B^+\sim_{\bQ} 0$. We say $(X, B)$ is \emph{$\bQ$-complementary} if it has at least one $\bQ$-complement.
    
    A $\bQ$-complementary pair $(X, B)$ is said to be \emph{exceptional} if every $\bQ$-complement of $K_X+B$ is klt.
\end{definition}

\begin{notation}
    Let $(X, B)$ be a klt pair. We put
    \begin{align*}
        \delta(X, B):=\#\{E| a(E, X, B)\leq \frac{1}{7}, E \text{ is an exceptional }\\
        \text{ or non-exceptional prime divisor over } X \}.
    \end{align*}
\end{notation}

\begin{theorem}[{\cite[Theorem 5.1]{Shokurov00Complements-on-surfaces}}]\label{thm:delta-leq-2}
    Let $(X, B)$ be an exceptional surface pair such that $B\in \Phi_{\mathrm{m}}$ and $-(K_X+B)$ is nef, then $\delta(X,B)\leq 2$.
\end{theorem}

The following lemma follows directly from the definitions.

\begin{lemma}\label{lem:crepant-preserve}
    Let $f\colon (Y, B_Y)\to (X, B)$ be a projective birational morphism between surface pairs such that $K_Y+B_Y=f^*(K_X+B)$. Then
    \begin{itemize}
        \item[(1)] $(X, B)$ is klt if and only if $(Y, B_Y)$ is. Furthermore, $\delta(X, B)=\delta(Y, B_Y)$.
        \item[(2)] $(X, B)$ is exceptional if and only if $(Y, B_Y)$ is. 
    \end{itemize}
\end{lemma}

We say that a surface $S$ has \emph{rank one} if $\rho(S)=1$.

\begin{theorem}[{\cite[Theorem 5.1.3]{Shokurov00Complements-on-surfaces},\cite[Theorem 10.2.1]{Prokhorov01lectures}}]\label{thm:rank-one-exceptional}
    Let $(S, \Delta)$ be a rank one log surface such that $K_S+\Delta$ is $\frac{1}{7}$-lt at closed points, $-(K_S+\Delta)$ is nef, $\Delta\in \Phi_{\mathrm{m}}$, $\delta(S, \Delta)=2$. Assume that $(S, \Delta)$ is exceptional. Write
    \begin{align*}
        \Delta=b_1C_1+b_2C_2+F, \quad F=\sum (1-\frac{1}{m_i})F_i\\
        b_1,b_2\geq \frac{6}{7}, m_i\in \{1,2,3,4,5,6\}.
    \end{align*}
    Then $C:=C_1+C_2$ has only normal crossings at smooth points of $S$, $\Supp(F)$ does not pass through $C_1\cap C_2$ and $b_1+b_2<\frac{13}{7}$. There is a complete classification of such pairs. If we assume, in addition, $F=0$, then the only possibilities are
    \begin{itemize}
        \item[\upshape{($A_2^6$)}] $S=\bP(1,2,3), \Delta=\frac{6}{7}(C_1+C_2)$, where $C_1$ is the line $\{x_1=0\}$, $C_2\in |-K_S|$, $\mathrm{Sing}(S)\subseteq C_1$; in this case $7(K_S+\Delta)\sim 0$.
        \item[\upshape{($I_2^2$)}] $S=\bP(1,2,3), \Delta=\frac{6}{7}(C_1+C_2)$, where $C_1=\{x_3=0\}$, $C_2=\{x_2^2-\alpha_1x_1^4-\alpha_2x_1^2x_2-x_1x_3=0\}$, $\alpha_1,\alpha_2\in \bC$, $(\alpha_1,\alpha_2)\neq (0,0)$; in this case $7(K_S+\Delta)\sim 0$.
    \end{itemize}
\end{theorem}

For the remainder of this paper, we may label the two boundary components on $S\simeq\mathbb P(1,2,3)$ by their weighted degrees. Thus
\begin{align*}
    C_1+C_2&=C_{(1)}+C_{(6)} \quad \text{in the case }(A_2^6),\\
    C_1+C_2&=C_{(3)}+C_{(4)} \quad \text{in the case }(I_2^2).
\end{align*}

For later use, we record convenient normal forms for the two rank-one exceptional surface pairs.

\begin{lemma}\label{lem:rank-one-normal-forms}
Suppose that the curve $C_{(6)}$ in the case $(A_2^6)$ is a nodal rational curve. Up to an automorphism of $\bP(1,2,3)$, the equations of curves in the cases $(A_2^6)$ and $(I_2^2)$ have the following forms:
\begin{enumerate}[label={\rm(\arabic*)}]
    \item In the case $(A_2^6)$,
    \begin{align*}
        C_{(1)}=(x_1=0),
        \qquad
        C_{(6)}=\left(x_3^2-(x_2-x_1^2)^2(x_2+2x_1^2)=0\right).
    \end{align*}
    Moreover, if we set $u:=x_2-x_1^2$, then we may write the equation of $C_{(6)}$ as
    \begin{align*}
        x_3^2-3x_1^2u^2-u^3=0.
    \end{align*}
    \item In the case $(I_2^2)$,
    \begin{align*}
        C_{(3)}=(x_3=0),
        \qquad
        C_{(4)}=\left(x_2^2-x_1^4-x_1x_3=0\right).
    \end{align*}
\end{enumerate}
\end{lemma}

\begin{proof}
We first consider the case $(A_2^6)$. By Theorem \ref{thm:rank-one-exceptional}, we may choose weighted coordinates $x_1,x_2,x_3$ of weights $1,2,3$ such that $C_{(1)}=(x_1=0)$ and $C_{(6)}$ is defined by
\begin{align*}
    ax_3^2+x_3Q_3(x_1,x_2)+Q_6(x_1,x_2)=0,
\end{align*}
where $Q_k$ is weighted homogeneous of degree $k$. By Theorem~\ref{thm:rank-one-exceptional}, both singular points of $S$ lie on $C_{(1)}$ and all intersection points of $C_{(1)}$ and $C_{(6)}$ lie in the smooth locus of $S$. It follows that the curve $C_{(6)}$ avoids the two singularities $[0:1:0]$ and $[0:0:1]$ and hence the coefficients of $x_3^2$ and $x_2^3$ in the above equation are nonzero.

After rescaling, we may assume that $a=1$. Replacing $x_3$ by $x_3-\frac{1}{2}Q_3(x_1,x_2)$ gives an equation $x_3^2=P_3(x_1^2,x_2)$, where $P_3$ is a binary cubic.

Since $Q_3(x_1,x_2)$ has weighted degree $3$ and $x_2$ has weight $2$, $Q_3(x_1,x_2)$ is divisible by $x_1$. Hence this coordinate change maps $[0:x_2:x_3]$ to $[0:x_2:x_3-Q_3(0,x_2)]=[0:x_2,x_3]$ in the new coordinates. Hence $(x_1=0)$ is invariant under this coordinate change, so the coordinates of two singularities remain $[0:1:0]$ and $[0:0:1]$. It follows that the coefficient of $x_2^3$ in $P_3$ is nonzero.

The singular points of the curve $x_3^2=P_3(x_1^2,x_2)$ correspond to the multiple roots of $P_3$. Since $C_{(6)}$ has a node, while a triple root would give a cusp, the binary cubic $P_3$ has one double root and one distinct simple root. Since the coefficient of $x_2^3$ in $P_3$ is nonzero, we may write
\begin{align*}
    P_3(x_1^2,x_2)=\lambda(x_2-\rho x_1^2)^2(x_2-\sigma x_1^2),
\end{align*}
for some $\lambda\neq0$ and distinct $\rho,\sigma \in \bC$. 

Choose
\begin{align*}
    c:=\frac{3}{\rho-\sigma},
    \qquad
    d:=1-c\rho
      =-\frac{2\rho+\sigma}{\rho-\sigma}.
\end{align*}
Then $c\neq 0$ and
\begin{align*}
    c\rho+d=1,
    \qquad
    c\sigma+d=-2.
\end{align*}
Replacing $x_2$ by $c^{-1}(x_2-dx_1^2)$ gives
\begin{align*}
    P_3\left(
        x_1^2,
        c^{-1}(x_2-dx_1^2)
    \right)=
    \lambda c^{-3}
    (x_2-x_1^2)^2
    (x_2+2x_1^2).
\end{align*}
Thus this weighted coordinate change sends the double root $[1:\rho:0]$ and the simple root $[1:\sigma:0]$ to $[1:1:0]$ and $[1:-2:0]$, respectively. The equation becomes
\begin{align*}
    x_3^2=(x_2-x_1^2)^2(x_2+2x_1^2).
\end{align*}
These coordinate changes leave $x_1=0$ unchanged and hence preserve $C_{(1)}$.

We next consider the case $(I_2^2)$. By Theorem \ref{thm:rank-one-exceptional}, we may choose weighted coordinates $x_1,x_2,x_3$ of weights $1,2,3$ such that
\begin{align*}
    C_{(3)}=(x_3=0), \quad C_{(4)}=\left(x_2^2-\alpha_2x_1^2x_2-\alpha_1x_1^4-x_1x_3=0\right),
\end{align*}
where $(\alpha_1,\alpha_2)\neq (0,0)$. By Theorem \ref{thm:rank-one-exceptional}, the intersection points of $C_{(3)}$ and $C_{(4)}$ lie on the smooth locus of $S$ and in fact on the chart $(x_1\neq 0)$. Thus these points correspond to the roots of
\begin{align*}
    t^2-\alpha_2t-\alpha_1=0,
    \qquad
    t:=\frac{x_2}{x_1^2}.
\end{align*}
Since these two points are distinct, $\alpha_2^2+4\alpha_1\neq0$. Set $\gamma:=\alpha_1+\frac{\alpha_2^2}{4}$.

Then $\gamma\neq0$. Replacing $x_2$ by $x_2-\frac{\alpha_2}{2}x_1^2$ transforms the equation of $C_{(4)}$ into
\begin{align*}
    x_2^2-\gamma x_1^4-x_1x_3=0.
\end{align*}
Finally, replacing $x_2,x_3$ by $\sqrt{\gamma}x_2, \gamma x_3$ and then dividing the equation by $\gamma$, gives $x_2^2-x_1^4-x_1x_3=0$. These coordinate changes leave $x_3=0$ unchanged and hence preserve $C_{(3)}$.
\end{proof}

\subsection{Toroidal surface pairs}

In this section, we collect the notation and basic facts on toric varieties and toroidal surface pairs. We refer the reader to \cite{CoxLittleSchenck11Toricvarieties} for toric varieties, and to \cite{Nakayama17toric} for toroidal surface pairs and toroidal morphisms. Throughout this paper, by a cone we mean a strongly convex rational polyhedral cone.

\begin{definition}[Toroidal pairs and toroidal morphisms, \upshape{\cite[Definitions~3.12, 3.13 and 4.19]{Nakayama17toric}}]\label{def:toroidal-local-model}
Let $X$ be a normal surface and $D$ a reduced divisor. 
\begin{itemize}
    \item[(1)] For a closed point $x$, the pair $(X, D)$ is said to be \emph{toroidal at} $x$ if there exists an affine toric variety $X_\sigma$ and two \'etale morphisms 
    \begin{align*}
        (X,x)\xleftarrow{\alpha} (V,v)
        \xrightarrow{\beta} (X_\sigma,t),\qquad t=\beta(v)
    \end{align*}
    with a point $v\in V$ lying over $x$ such that $\alpha^{-1}(X\setminus D)=\beta^{-1}(T_\sigma)$ for the open torus $T_\sigma$ of $X_\sigma$. We call the pair $(X_\sigma,t)$ a \emph{local model} of $(X, D)$ at $x$.
    \item[(2)] The pair $(X, D)$ is said to be \emph{toroidal along} a subset $Z$ of $X$ if $(X, D)$ is toroidal at each closed point of $Z$. If $(X, D)$ is toroidal along $X$, then $(X, D)$ is said to be \emph{toroidal}.
    \item[(3)] An \emph{\'etale neighbourhood} of $(X, x)$ is a pair $(V, v)$ of a variety $V$ and a closed point $v\in V$ together with an \'etale morphism $\alpha\colon V\to X$ such that $x=\alpha(v)$ and $\alpha$ induces an isomorphism $k(x)\simeq k(v)$ of residue fields. Since we work over $\bC$, the last condition holds automatically.
    \item[(4)] A proper birational morphism $h\colon Y\to X$ from another normal surface $Y$ is called a \emph{toroidal morphism}(a \emph{toroidal blowing up} in Nakayama's terminology)) with respect to $(X, D)$ if the following conditions are satisfied:
        \begin{itemize}
            \item $h(\Exc(h))\subseteq D$;
            \item $(X, D)$ is toroidal along $h(\Exc(h))$, and $(Y, D_Y)$ is toroidal along $\Exc(h)$, where $D_Y:=h_*^{-1}D+\Exc(h)$;
            \item $K_Y+D_Y=h^*(K_X+D)$.
        \end{itemize}
\end{itemize}
\end{definition}

\begin{proposition}[\upshape{\cite[Proposition~4.21(1)(4)]{Nakayama17toric}}]\label{prop:toroidal-local-description}
Let $(X,D)$ be a normal toroidal surface pair, and let $h\colon Y\to X$ be a proper birational morphism from a normal surface $Y$. Assume that $h$ is an isomorphism over $X\setminus\Supp D$, and set $D_Y:=h_*^{-1}D+\Exc(h)$. Then the following conditions are equivalent:
\begin{enumerate}
    \item[(1)] The morphism $h\colon (Y,D_Y)\to (X,D)$ is toroidal.

    \item[(2)] For every point $y\in D_Y$, setting $P:=h(y)$, there exist
    \begin{itemize}
        \item an \'etale neighborhood $\alpha\colon (V,v)\to (X,P)$;
        \item a proper birational toric morphism $\varphi\colon Z'\to Z$ of two dimensional toric surfaces;
        \item an \'etale morphism $\beta\colon V\to Z$ such that
        \[
            \alpha^{-1}(D)=\beta^{-1}(D_Z)
        \]
        as reduced divisors, where $D_Z$ denotes the toric boundary of $Z$;
        \item a Cartesian diagram
        \[
        \begin{CD}
            Y_V @>{h_V}>> V\\
            @V{\gamma}VV @VV{\beta}V\\
            Z' @>{\varphi}>> Z,
        \end{CD}
        \]
        where $Y_V:=Y\times_X V$ and $h_V\colon Y_V\to V$ is the base change of $h$.
    \end{itemize}
\end{enumerate}
\end{proposition}

\begin{remark}\label{rem:toroidal--computations}
By Proposition~\ref{prop:toroidal-local-description}, a toroidal birational morphism $h\colon Y\to X$ of normal surfaces is, \'etale locally on $X$, obtained by base change from a proper birational toric morphism of toric surfaces. In particular, its local model over a zero-dimensional stratum $x$ of $(X,D)$ is described by a subdivision of the corresponding two-dimensional cone.

Consequently, any computation depending only on the local structure of $(X, B)$ over $x$ may be carried out in the corresponding local model; for example, one may compute log discrepancies, intersection numbers involving exceptional curves over $x$, and orders of functions or sections along these curves.
\end{remark}

We now make this local description explicit and fix the notation used below.

\begin{notation}\label{not:local-toric}
Fix a zero-dimensional stratum $P$ of a toroidal surface pair $(X,D)$. Choose a local model $(X_\sigma, t)$ at $P$, together with a common \'etale neighborhood
\begin{align*}
    (X,P)\xleftarrow{\alpha}(V,v)
    \xrightarrow{\beta}(X_\sigma,t).
\end{align*}
Since $P$ is a zero-dimensional stratum, $t$ is the torus-fixed point of $X_\sigma$ corresponding to the two-dimensional cone $\sigma$. After shrinking $V$ around $v$, we may assume that $V$ is irreducible and $\alpha^{-1}(P)=\beta^{-1}(t)=\{v\}$.

After choosing an identification $N\simeq\mathbb Z^2$, \cite[Proposition~10.1.1]{CoxLittleSchenck11Toricvarieties} allows us to write
\begin{align*}
    \sigma =\bR_{\geq0}u_0+\bR_{\geq0}u_{n+1}, \quad u_0=(r,-s), u_{n+1}=(0,1)
\end{align*}
where $0\leq s<r, \gcd(r,s)=1$. Let $\cD_0$ and $\cD_{n+1}$ be the invariant prime divisors on $X_\sigma$ corresponding to $\bR_{\geq 0}u_0$ and $\bR_{\geq 0}u_{n+1}$, respectively. Under the chosen local model, these divisors correspond to the two analytic branches of $D$ at $P$. These two branches may come from the same irreducible curve on $X$.

Let $u_i=(p_i,q_i), ~1\leq i\leq n$ be primitive vectors in the interior of $\sigma$, ordered by increasing slope. Let $\Sigma$ be the subdivision of $\sigma$ obtained by adding the rays $\mathbb R_{\geq0}u_i$. For each $i$, let $E_i$ denote the prime divisor over $X$ corresponding, via the chosen local model, to the torus-invariant prime divisor associated with the ray $\mathbb R_{\geq0}u_i$.
\end{notation}

\begin{lemma}\label{lem:local-toric-discrepancy}
Notation as in Notation~\ref{not:local-toric}. Let $\Delta$ be a boundary on $X$ supported on $D$. Suppose that, in the chosen local model, the germ of $\Delta$ at $P$ corresponds to
\begin{align*}
    b_0\mathcal D_0+b_{n+1}\mathcal D_{n+1},
    \qquad
    0\leq b_0,b_{n+1}\leq1.
\end{align*}
Let $u=(p,q)$ be a primitive vector in the interior of $\sigma$, and $E$ be the prime divisor over $X$ determined by the ray $\mathbb R_{\geq0}u$. Then
\begin{align}\label{eq:local-toric-discrepancy}
    a(E,X,\Delta)=\frac{p}{r}(1-b_0) +\frac{sp+rq}{r}(1-b_{n+1}).
\end{align}
\end{lemma}

\begin{proof}
By Remark~\ref{rem:toroidal--computations}, it suffices to compute the log discrepancy in the corresponding local model. Hence $a(E,X,\Delta)$ is the value at $u$ of the linear function $\psi$ on $\sigma$ determined by
\begin{align*}
    \phi(u_0)=1-b_0, \qquad \phi(u_{n+1})=1-b_{n+1},
\end{align*}
see \cite[Section~2, p.~360]{Ambro06toricmld}. Since
\begin{align*}
    (p,q)=\frac{p}{r}(r,-s)+\frac{sp+rq}{r}(0,1),
\end{align*}
formula \eqref{eq:local-toric-discrepancy} follows by linearity.
\end{proof}

We next record the intersection number formulas for the exceptional curves of a toroidal morphism over a zero-dimensional stratum.

\begin{lemma}\label{lem:toric-intersections}
Notation as in Notation~\ref{not:local-toric}. Let $h\colon Y\to X$ be a toroidal birational morphism whose local toric model over $P$ is $X_\Sigma\longrightarrow X_\sigma$. Assume that $n\geq1$. Then $E_1,\ldots,E_n$ are precisely the $h$-exceptional prime divisors centered at $P$, and
\begin{align*}
    E_i^2&=-\frac{\det(u_{i-1},u_{i+1})}{\det(u_{i-1},u_i)\det(u_i,u_{i+1})}, &&1\leq i\leq n,\\
    E_i\cdot E_{i+1}&=\frac{1}{\det(u_i,u_{i+1})}, &&1\leq i\leq n-1,\\
    E_i\cdot E_j&=0, &&|i-j|\geq2.
\end{align*}
Here the determinants are computed with respect to the chosen identification $N\simeq\mathbb Z^2$.

Moreover, let $D_0$ and $D_{n+1}$ be the two analytic branches of $D$ at $P$ corresponding to $\cD_0$ and $\cD_{n+1}$, respectively. let $D_{0,Y}$ and $D_{n+1,Y}$ denote their strict transforms near $h^{-1}(P)$. Then
\begin{align*}
    D_{0,Y}\cap\Exc(h)&=D_{0,Y}\cap E_1=\{Q_0\},\\
    D_{n+1,Y}\cap\Exc(h)&=D_{n+1,Y}\cap E_n=\{Q_{n+1}\},
\end{align*}
and
\begin{align*}
    (D_{0,Y}\cdot E_1)_{Q_0}=\frac{1}{\det(u_0,u_1)},\qquad
    (D_{n+1,Y}\cdot E_n)_{Q_{n+1}}=\frac{1}{\det(u_n,u_{n+1})}.
\end{align*}
\end{lemma}

\begin{proof}
We first compute on the local model  $X_\Sigma$. The stated formulas are the standard intersection formulas for torus-invariant curves on a toric surface; see \cite[Proposition~2.2 and Lemma~2.3]{Shao26p} or \cite[Section~6.4]{CoxLittleSchenck11Toricvarieties}. Nonadjacent rays do not belong to a common two-dimensional cone, so the corresponding invariant curves are disjoint.

By Proposition~\ref{prop:toroidal-local-description}, we may choose a common étale neighbourhood $(V,v)$ as in Notation~\ref{not:local-toric} such that
\begin{align*}
    Y_V:=Y\times_X V
    \simeq X_\Sigma\times_{X_\sigma}V.
\end{align*}
Since $\alpha^{-1}(P)=\beta^{-1}(t)=\{v\}$, the projections $Y_V\to Y$ and $Y_V\to X_\Sigma$ induce isomorphisms between the corresponding exceptional curves. Their pullbacks agree as curves on $Y_V$, so the intersection formulas also hold on $Y$.

The assertions concerning the two analytic boundary branches follow from the adjacency of the rays and the same local intersection formulas.
\end{proof}

The next lemma shows that a crepant birational morphism under certain assumptions is a toroidal morphism, and hence can be described by subdivisions in compatible local models.

\begin{lemma}\label{lem:crepant-toroidal}
Let $(X,D)$ be a normal toroidal surface pair, and let $\Delta$ be a $\bQ$-boundary supported on $D$. Let $h\colon Y\to X$ be a projective birational morphism between normal surfaces. Assume that 
\begin{itemize}
    \item[(1)] $(X,\Delta)$ is klt and $X\setminus\Supp(\Delta)$ is smooth;
    \item[(2)] $K_Y+\Delta_Y=h^*(K_X+\Delta)$, where $\Delta_Y:=h_*^{-1}\Delta$.
\end{itemize}
Then $h$ is an isomorphism over $X\setminus \Supp(\Delta)$. Moreover, $(Y,D_Y)$ is toroidal and $h\colon(Y,D_Y)\to (X,D)$ is a toroidal morphism, where $D_Y:=h_*^{-1}D+\Exc(h)$.
\end{lemma}

\begin{proof}
We first show that every $h$-exceptional curve corresponds to a ray in a local toric model of $(X,D)$ at its center. Fix a prime $h$-exceptional curve $E$, and let $x:=c_X(E)$. Since $E$ does not occur in $\Delta_Y$, Condition (2) gives $a(E,X,\Delta)=1$.

Suppose that $x\notin \Supp(\Delta)$. By Condition~{\rm(1)}, $x$ is a smooth point of $X$, and $\Delta$ vanishes near $x$. Hence
\begin{align*}
    a(E,X,\Delta)=a(E,X,0)\geq 2,
\end{align*}
a contradiction. Thus every prime $h$-exceptional curve is centered in $\Supp(\Delta)$. Consequently, $h$ is an isomorphism over $X\setminus \Supp(\Delta)$.

Choose a toroidal log resolution $\pi\colon (X_0,D_0)\to (X,D)$, where $D_0:=\pi_{*}^{-1}D+\Exc(\pi)$. Thus $X_0$ is smooth and $D_0$ is a simple normal crossings divisor, and, in compatible local models, $\pi$ is induced by regular subdivisions of the corresponding cones.

Define $\Delta_0$ by $K_{X_0}+\Delta_0=\pi^*(K_X+\Delta)$. Then $\Supp(\Delta_0)\subseteq D_0$ and every coefficient of $\Delta_0$ is strictly less than $1$ because $(X,\Delta)$ is klt.

If $E$ appears as a prime divisor on $X_0$, then $E$ is a component of $D_0$ and hence corresponds to a ray in a local model of $(X,D)$ at $x$. We may therefore assume that $E$ does not appear on $X_0$.

Extract $E$ by successively blowing up its centers. Thus there is a sequence of blow-ups of closed points
\begin{align*}
    X_n\xrightarrow{\mu_{n-1}}X_{n-1}
    \longrightarrow\cdots\longrightarrow
    X_1\xrightarrow{\mu_0}X_0,
\end{align*}
where $\mu_i$ is the blow-up of the center $x_i:=c_{X_i}(E)$ on $X_i$, and the exceptional divisor of $\mu_{n-1}$ is $E$.

For each $i$, let $\varphi_i\colon X_i\to X_0$ be the induced morphism, and set
\begin{align*}
    D_i:=(\varphi_{i})_{*}^{-1}D_0+\Exc(\varphi_i),\qquad
    K_{X_i}+\Delta_i:= \varphi_i^*(K_{X_0}+\Delta_0).
\end{align*}
Then $D_i$ is a simple normal crossing divisor, $\Supp(\Delta_i)\subseteq D_i$, and every coefficient of $\Delta_i$ is strictly less than $1$.

Suppose that some $x_i$ is not an intersection point of two components of $D_i$, and choose the smallest such $i$. Since $D_i$ is a simple normal crossing divisor, at most one component of $D_i$ passes through $x_i$. Let $d_i$ be its coefficient in $\Delta_i$, and set $d_i=0$ if $x_i\notin D_i$. If $F_{i+1}$ denotes the exceptional divisor of $\mu_i$, then 
\begin{align}\label{eq:first-negative-crepant-coefficient}
    \mathrm{mult}_{F_{i+1}}(\Delta_{i+1})
    =d_i-1<0.
\end{align}

We claim that $\mult_{F_{j}}(\Delta_j)<0$ for every $i+1\leq j\leq n$. The case $j=i+1$ follows from \eqref{eq:first-negative-crepant-coefficient}. Suppose that $i+1\leq j<n$ and $\mult_{F_j}(\Delta_j)<0$. As long as $E$ has not yet been extracted, $x_j\in F_j$. Besides $F_j$, at most one further component of $D_j$ passes through $x_j$. Let $d_j<1$ be its coefficient in $\Delta_j$, and set $d_j=0$ if no such component exists. Then
\begin{align*}
    \mult_{F_{j+1}}(\Delta_{j+1})=\mult_{F_{j}}(\Delta_{j})+d_j-1<0.
\end{align*}
This proves the claim.

Since $E=F_n$, the claim gives $\mult_{E}(\Delta_n)=\mult_{F_{n}}(\Delta_n)<0$. On the other hand,
\begin{align*}
    \mult_E(\Delta_n)=1-a(E,X,\Delta)=0,
\end{align*}
a contradiction. Therefore every $x_i$ is an intersection point of two components of $D_i$. Hence every $\mu_i$ is toroidal. Since $E$ was arbitrary, every prime $h$-exceptional divisor corresponds to a ray in the local model of $(X,D)$ at its center. 

Fix $x\in h(\Exc(h))$. Let $\Sigma$ be the subdivision of the corresponding two-dimensional cone obtained by adding the rays associated with the $h$-exceptional prime divisors centered at $x$. After base change to a common \'etale neighborhood of $x$, the associated toric morphism $X_\Sigma\to X_\sigma$ has the same exceptional prime divisors as $h$. The induced birational map between the two surfaces over $x$ is therefore an isomorphism in codimension one, hence an isomorphism, since normal surfaces admit no nontrivial small birational maps.

Thus, \'etale locally at every point of $h(\Exc(h))$, the morphism $h$ is obtained by base change from a proper birational toric morphism. By Proposition~\ref{prop:toroidal-local-description}, $(Y,D_Y)$ is toroidal and $h\colon (Y,D_Y)\to(X,D)$ is a toroidal morphism.
\end{proof}

\begin{definition}\label{def:orbifold-parameters}
Fix the local model of Notation~\ref{not:local-toric}. Via the common \'etale neighborhood, we identify the analytic germ of the toroidal pair at the distinguished point with the analytic germ at the torus-fixed point of
\begin{align*}
    \left(
        \bA^2_{z_1,z_2}/\mu_r(1,s),0
    \right),
    \qquad
    r\geq1,\quad
    0\leq s<r,\quad
    \gcd(r,s)=1,
\end{align*}
so that the branches of two ordered torus-invariant curves are the images of $(z_1=0)$ and $(z_2=0)$. We allow $r=1$ and $s=0$.

An ordered pair $(f_1,f_2)$ of analytic functions on the orbifold cover is called a system of \emph{ordered boundary-adapted orbifold parameters} if $(f_1,f_2)$ is a regular system of parameters at the origin,
\begin{align*}
    \zeta^*f_1=\zeta f_1,
    \qquad
    \zeta^*f_2=\zeta^s f_2
    \qquad
    (\zeta\in\mu_r),
\end{align*}
and $(f_1=0)=(z_1=0),\qquad (f_2=0)=(z_2=0)$ as reduced analytic germs.

For brevity, we call $(f_1,f_2)$ \emph{orbifold parameters}. When $r=1$ and $s=0$, we simply call $(f_1,f_2)$ \emph{analytic coordinates}.
\end{definition}

The following lemma computes the order of a section in orbifold parameters.

\begin{lemma}\label{lem:toric-orders}
Notation as in Notation \ref{not:local-toric}. Let $E$ be the divisor corresponding to an interior primitive vector $u=(p,q)$ of $\sigma\cap N$ in the fixed local model. Let $(f_1,f_2)$ be orbifold parameters of type $\frac{1}{r}(1,s)$. Let $A$ be a $\bQ$-Cartier Weil divisor near $P$ and let $\tau$ be a nonzero local section of $\cO_X(A)$. Suppose that, on the orbifold cover, $\tau$ is represented by a $\mu_r$-semi-invariant function of degree $d$
\begin{align*}
    f_d=\sum_{\substack{a_1,a_2\geq0\\
a_1+sa_2\equiv d\;(\mathrm{mod}\;r)}}
c_{a_1,a_2}f_1^{a_1}f_2^{a_2}.
\end{align*}
Then 
\begin{align}\label{eq:toric-order-formula}
    \ord_{E}(\tau)=\min \big\{\frac{a_1p+a_2(sp+rq)}{r}|c_{a_1,a_2}\neq0 \big\}.
\end{align}
In particular, when $r=1$ and $s=0$, 
\begin{align*}
    \ord_{E}(\tau)=\min \{a_1p+a_2q| c_{a_1,a_2}\neq0\}.
\end{align*}
\end{lemma}

\begin{proof}
In the fixed local model, we have
\[
    u=(p,q)
    =\frac{p}{r}(r,-s)+\frac{sp+rq}{r}(0,1).
\]
Thus the subdivision obtained by adding the ray $\mathbb R_{\geq0}u$ is the weighted blow-up with weights
\[
    \operatorname{wt}(z_1,z_2)
    =\frac1r(p,sp+rq);
\]
see \cite[Definition~2.2.11]{Kawakita24book}. Its exceptional divisor corresponds to $E$.

Since $(f_1,f_2)$ defines the same ordered branches as $(z_1,z_2)$, we have $f_i=\varepsilon_i z_i$ for $\mu_r$-invariant units $\varepsilon_i$. The corresponding invariant monomials generate the same principal ideals. Thus the weighted ideals defining the weighted blow-up are unchanged, and we may compute the order along $E$ in $(f_1,f_2)$.

Applying the weighted-order formula following \cite[pp.~70--71]{Kawakita24book} to the semi-invariant representative $f_d$ of $\tau$, we obtain
\begin{align*}
    \operatorname{ord}_{E}(\tau)
    =\operatorname{w\text{-}ord}(f_d)
    =\min \big\{
    \frac{a_1p+a_2(sp+rq)}{r}|c_{a_1,a_2}\neq0\big\}.
\end{align*}
This proves \eqref{eq:toric-order-formula}.
\end{proof}

\section{A lower bound}
\label{sec:lower-bound}

\begin{theorem}\label{thm:lower-bound}
Let $(X,B)$ be a normal projective lc surface pair such that $K_X+B$ is ample and $B$ is a nonzero reduced divisor.  Suppose that $(X,B)$ has a zero-dimensional non-klt center $x\in\Supp(B)$. Then
\begin{align*}
    (K_X+B)^2\geq\frac{1}{42}.
\end{align*}
\end{theorem}

\begin{proof}
Fix an irreducible component $B_0$ of $B$ through $x$. By Lemma \ref{lem:lc -pair-underlying-klt}, $X$ is klt along $\Supp(B)$, hence it is $\bQ$-factorial there. In particular, $B$ is $\bQ$-Cartier. Set
\begin{align*}
     b:=\pet(X,0;B).
\end{align*}

Since $(X,B)$ is log canonical, $(X,tB)$ is lc for every $t\in[0,1]$. The openness of the ample cone gives $b<1$, and the closedness of the pseudo-effective cone shows that $K_X+bB$ is pseudo-effective.

Lemma~\ref{lem:positive-branch-one-sixth} gives
\begin{align}
    (K_X+B)\cdot B_0\geq\frac{1}{6}.
\end{align}

Suppose first that $b\leq\frac{6}{7}$. Since $K_X+B$ is nef and $K_X+bB$ is pseudo-effective,
\begin{align}\label{eq:1/42}
    (K_X+B)^2
    &=(K_X+B)\cdot(K_X+bB)+(1-b)(K_X+B)\cdot B\notag\\
    &\geq(1-b)(K_X+B)\cdot B
     \geq\frac{1}{7}(K_X+B)\cdot B_0
     \geq\frac{1}{42}.
\end{align}

Now we assume that $b>\frac{6}{7}$ and will show that, in this case, $(K_X+B)^2\geq \frac{1}{22}$. Choose a rational number $s$ with $\frac{6}{7}<s<b$. Run a $(K_X+sB)$-MMP, which terminates with a Mori fiber space $f\colon W\longrightarrow V$. Let $h\colon X\longrightarrow W$ be the induced morphism and let $B_W:=h_*B$. Since $K_X+B$ is ample, $K_W+B_W$ is also ample. Note that $h$ may be an identity.

Since $(W,sB_W)$ is lc and $s>\frac{5}{6}$, Lemma~\ref{lem:surface-one-gap} shows that $(W, B_W)$ is log canonical. Since $K_W+bB_W$ is nef over $V$ and $K_W+sB_W$ is anti-ample over $V$, the divisor $B_W$ is ample over $V$. In particular, $B_W\neq0$ and there is a component of $B_W$ that is horizontal over $V$. Since, in addition, $\rho(W/V)=1$, there is a number $b_W\in(s,b]$ such that $K_W+b_WB_W\equiv_V 0$. Thus \cite[Corollary~7.1]{LiuShokurov2023optimal} implies that $V$ is a point. Consequently
\begin{align}
    K_W+b_WB_W\equiv0, \qquad \rho(W)=1.
\end{align}

We next show that $h$ is an isomorphism near $x$. Suppose that a curve $C$ through $x$ is contracted by $h$. Since $K_X+B$ is ample, the negativity lemma implies that we can write
\begin{align*}
    h^*(K_W+B_W)=K_X+B+aC+D,
\end{align*}
where $a>0$ and $D$ is an effective divisor supported on $\mathrm{Exc}(h)$. Let $F$ be a lc place of $(X, B)$ centered at $x$. Then
\begin{align*}
    a(F, W, B_W)&=a(F, X, B+aC+D)\\
    &=a(F, X, B)-a\ord_F(C)-\ord_F(D)\\
    &\leq -a\ord_F(C)<0.
\end{align*}
This is a contradiction. Hence no curve through $x$ is contracted by $h$ and $h$ is an isomorphism near $x$. A similar computation shows that no component of $B$ is contracted by $h$. Note that $h(x)$ remains a zero-dimensional non-klt center of $(W,B_W)$.

Apply \cite[Lemma~3.11]{LiuShokurov2023optimal} to $(W, b_WB_W)$.  Since $b_W>\frac{6}{7}$, the divisor $B_W$ is irreducible and exactly one of the following occurs:
\begin{enumerate}[label=\textup{(\roman*)}]
    \item $B_W$ is a smooth rational curve;
    \item $B_W$ is a smooth elliptic curve;
    \item $B_W$ is a rational curve with one node, and that node is a
    smooth point of $W$.
\end{enumerate}

In the first two cases, the same lemma implies that $(W,B_W)$ is plt near $\mathrm{Sing}(W)$, while at smooth points this is immediate. This contradicts the existence of a zero-dimensional non-klt center contained in $B_W$. Therefore the third case occurs, and $h(x)$ is the unique node of $B_W$ and is a smooth point of $W$.

Since no component of $B$ is contracted by $h$ and $B_W$ is irreducible, $B$ is also irreducible and its normalization is $\mathbb P^1$. Since $h$ is an isomorphism near $x$, $x$ is a node of $B$ and $x$ is a smooth point of $X$. The two inverse images of $x$ in the normalization both have coefficient $1$ in the different.  Adjunction formula therefore gives
\begin{align*}
    (K_X+B)\cdot B
    =-2+1+1+\sum_{j=1}^m\frac{q_j-1}{q_j}.
\end{align*}
Since $(K_X+B)\cdot B>0$, we have $(K_X+B)\cdot B\geq \frac{1}{2}$.

Applying \cite[Theorem~2.6]{LiuLiu26minimal} to $(X, B)$, we have
$b=\frac{12}{13}$ or $b\leq\frac{10}{11}$. In the first case the same theorem says that $B$ is a smooth rational curve, contrary to the node at $x$. Hence $b\leq\frac{10}{11}$, and
\begin{align*}
    (K_X+B)^2&=(K_X+B)\cdot(K_X+bB)+(1-b)(K_X+B)\cdot B\\
    &\geq(1-b)(K_X+B)\cdot B
     \geq\frac{1}{11}\cdot\frac{1}{2}
     =\frac{1}{22}.
\end{align*}
Together with \eqref{eq:1/42}, this completes the proof.
\end{proof}

\section{An example with the minimal volume}
\label{sec:construction}

In this section, we construct an example to show that the lower bound $\frac{1}{42}$ is sharp.

\begin{proposition}\label{prop:construction}
There is a surface pair $(X, B)\in \cS'(2, \{0,1\})$ such that
\begin{align*}
    (K_X+B)^2=\frac{1}{42}.
\end{align*}
\end{proposition}

\begin{proof}
We take the member $(S, \frac{6}{7}(C_1+C_2))$ of type $(I_2^2)$ as in Theorem~\ref{thm:rank-one-exceptional} with $\alpha_1=1, \alpha_2=0$. Thus
\begin{align*}
    C_1=\{x_3=0\},\quad 
    C_2=\{x_2^2=x_1^4+x_1x_3\}.
\end{align*}
The curves $C_1$ and $C_2$ meet transversely at the two smooth points $[1:1:0]$ and $[1:-1:0]$. Let $H$ denote the class of $\mathcal O_S(1)$; then $H^2=\frac16$. Since
\begin{align*}
    K_S\sim-6H,
    \qquad C_1\sim3H,
    \qquad C_2\sim4H,
\end{align*}
we have
\begin{align}\label{eq:calabi-yau-source}
    K_S+\frac{6}{7}(C_1+C_2)\sim_{\bQ}0,
    \qquad K_S+C_1+C_2\sim H,
    \qquad (K_S+C_1+C_2)^2=\frac{1}{6}.
\end{align}

Fix one of these two points, say $P=[1:1:0]$. Choose a local model $(X_\sigma,t)$ at $P$ such that the primitive generators of the cone $\sigma$ are $u_0=(1,0)$ and $u_3=(0,1)$, where $u_0$ and $u_3$ correspond to $C_1$ and $C_2$, respectively. Subdivide this cone by adding the rays generated by $u_1=(5,2), u_2=(2,5)$.

\begin{figure}[ht]
\centering
\begin{tikzpicture}[scale=0.6,>=Latex]
    \coordinate (O) at (0,0);

    \draw[very thick,->]
        (O)--(3,0)
        node[right] {$u_0=(1,0)$};

    \draw[very thick,->]
        (O)--(0,3)
        node[above] {$u_3=(0,1)$};

    \draw[red!70!black,dashed,very thick,->]
        (O)--(2.5,1)
        node[right] {$u_1=(5,2)$};

    \draw[red!70!black,dashed,very thick,->]
        (O)--(1,2.5)
        node[above right] {$u_2=(2,5)$};

    \fill (O) circle (1.4pt);
\end{tikzpicture}
\caption{The fan subdivision at the local model of $P$.}
\label{fig:equality-fan}
\end{figure}
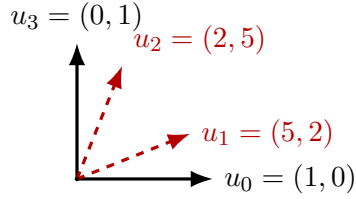

Let $h\colon Y\to S$ be the projective birational morphism obtai
ned from the fan subdivision. Denote by $F_1,F_2$ the exceptional curves corresponding to the rays $\bR_{\geq 0}u_1, \bR_{\geq 0}u_2$, respectively. The morphism $h$ is an isomorphism away from $P$. Let $\tilde C_1,\tilde C_2$ be the strict transforms of $C_1,C_2$, respectively.

Applying Lemma~\ref{lem:local-toric-discrepancy} to $F_1,F_2$, we obtain
\begin{align*}
    a\left(
        F_i,S,\frac{6}{7}(C_1+C_2)
    \right)&=1,
    &
    a(F_i,S,C_1+C_2)&=0
    \qquad (i=1,2).
\end{align*}
Hence
\begin{align}\label{eq:crepant-six-sevenths}
    K_Y+\frac{6}{7}(\tilde{C}_1+\tilde{C}_2)
    &=
    h^*\left(
        K_S+\frac{6}{7}(C_1+C_2)
    \right)
    \sim_{\bQ}0
\end{align}
and
\begin{align}\label{eq:pullback-log-canonical}
    K_Y+\tilde{C}_1+\tilde{C}_2
    =
    h^*(K_S+C_1+C_2)-F_1-F_2.
\end{align}
Moreover, $(Y,\tilde{C}_1+\tilde{C}_2)$ is lc.

Lemma~\ref{lem:toric-intersections} gives
\begin{align}\label{eq:G-intersections}
    F_1^2=-\frac{5}{42},\quad
    F_2^2=-\frac{5}{42},\quad
    F_1\cdot F_2=\frac{1}{21},\quad
    F_1\cdot\tilde C_1=\frac{1}{2},\quad
    F_2\cdot\tilde C_2=\frac{1}{2}.
\end{align}
We also have $F_1\cdot \tilde{C}_2=F_2\cdot \tilde C_1=0$. In particular,
\begin{align*}
    (F_1+F_2)^2
    &=
    F_1^2+2F_1\cdot F_2+F_2^2\\
    &=
    -\frac{5}{42}
    +\frac{2}{21}
    -\frac{5}{42}
    =
    -\frac{1}{7}.
\end{align*}
Combining this with \eqref{eq:pullback-log-canonical}, we obtain
\begin{align}\label{eq:volume-computation}
    (K_Y+\tilde{C}_1+\tilde{C}_2)^2
    &=
    (K_S+C_1+C_2)^2+(F_1+F_2)^2\notag\\
    &=
    \frac{1}{6}-\frac{1}{7}
    =
    \frac{1}{42}.
\end{align}

By \ref{eq:calabi-yau-source} and \ref{eq:pullback-log-canonical}, we have
\begin{align}\label{eq:all-intersections}
    (K_Y+\tilde{C}_1+\tilde{C}_2)\cdot\tilde C_1=0,\qquad
    (K_Y+\tilde{C}_1+\tilde{C}_2)\cdot\tilde C_2=\frac{1}{6}.
\end{align}
By \eqref{eq:crepant-six-sevenths}, we have $K_Y+\tilde{C}_1+\tilde{C}_2\sim_{\bQ}\dfrac{1}{7}(\tilde{C}_1+\tilde{C}_2)$. Hence $(K_Y+\tilde{C}_1+\tilde{C}_2)\cdot C\geq 0$ for every irreducible curve $C\neq \tilde{C}_1,\tilde{C}_2$. Combine this with \eqref{eq:all-intersections}, it follows that $K_Y+\tilde{C}_1+\tilde{C}_2$ is nef. Since $(K_Y+\tilde{C}_1+\tilde{C}_2)^2>0$, it is also big.

By abundance for log canonical surfaces \cite[Theorem~1.1]{Fujino12MMP}, the divisor $K_Y+\tilde{C}_1+\tilde{C}_2$ is semiample. Let $g\colon Y\to X$ be the morphism defined by $K_Y+\tilde{C}_1+\tilde{C}_2$. Since $K_Y+\tilde{C}_1+\tilde{C}_2$ is big, $g$ is birational. Set $B:=g_*(\tilde{C}_1+\tilde{C}_2)$. By \eqref{eq:all-intersections}, the morphism $g$ contracts $\tilde C_1$ but does not contract $\tilde C_2$. It may also contract other curves. Thus $B=g_*\tilde C_2$ is a nonzero reduced curve. By the definition of $g$, we have
\begin{align}\label{eq:ample-model-crepant}
    K_Y+\tilde{C}_1+\tilde{C}_2=g^*(K_X+B),
\end{align}
where $K_X+B$ is ample. Since $(Y, \tilde{C}_1+\tilde{C}_2)$ is lc, so is $(X,B)$.

By construction, $h$ is an isomorphism over the other point $[1:-1:0]$ of $C_1\cap C_2$. Hence $\tilde C_1\cap\tilde C_2\neq\varnothing$. Let $x:=g(\tilde C_1)$. Then $x\in B$. Equation \eqref{eq:ample-model-crepant} shows that $\tilde C_1$ is an lc place centered at $x$ and $x$ is a non-klt center of $(X, B)$. Therefore $(X, B)\in \cS'(2, \{0,1\})$. Finally,
\begin{align*}
    (K_X+B)^2=(K_Y+\tilde{C}_1+\tilde{C}_2)^2=\frac{1}{42}.
\end{align*}
\end{proof}

\section{Some properties of equality pairs}
\label{sec:equality-pair-properties}

We call a pair $(X, B)$ in $\cS'(2,\{0,1\})$ of volume $\frac{1}{42}$ an \emph{equality pair}. In this section, we establish some properties of equality pairs that will be used in the next section.

\begin{theorem}\label{thm:equality-properties}
Let $(X,B)$ be an equality pair, and let $x$ be the zero-dimensional non-klt center on the boundary. Then the following statements hold.
\begin{enumerate}[label={\rm(\arabic*)}]
    \item $K_X+\frac{6}{7}B\sim_{\bQ}0$ and $(K_X+B)\cdot B=\frac{1}{6}$.
    \item $B$ is a smooth rational curve, and $\mathrm{Diff}_B(0)=x+\frac{1}{2}p_2+\frac{2}{3}p_3$. Moreover, $x$ is the unique zero-dimensional non-klt center of $(X, B)$. 
    \item $(X, \frac{6}{7}B)$ is an exceptional log Calabi-Yau pair and $\delta(X, \frac{6}{7}B)=2$.
\end{enumerate}
\end{theorem}

\begin{proof}[Proof of Theorem \ref{thm:equality-properties} (1) and (2)]
The case $\pet(X, 0; B)>\frac{6}{7}$ in the proof of Theorem \ref{thm:lower-bound} gives $(K_X+B)^2\geq \frac{1}{22}$, so the equality pair satisfies $\pet(X, 0; B)\leq \frac{6}{7}$.

Let $B_0$ be the irreducible component of $B$ containing $x$. Since $(K_X+B)^2=\frac{1}{42}$, equality holds throughout \eqref{eq:1/42}. Hence
\begin{align*}
    &\pet(X, 0;B)=\frac{6}{7}, &\quad &(K_X+B)\cdot B_0=\frac{1}{6},\\
    &(K_X+B)\cdot \left(K_X+\frac{6}{7}B\right)=0, &\quad &(K_X+B)\cdot (B-B_0)=0
\end{align*}
As $K_X+B$ is ample, the last equality implies $B=B_0$ and therefore $B$ is irreducible. Moreover, $K_X+\frac{6}{7}B$ is pseudo-effective, and its intersection with the ample divisor $K_X+B$ is zero. Thus $K_X+\frac{6}{7}B\equiv0$. Abundance for lc surfaces gives $K_X+\frac{6}{7}B\sim_{\bQ}0$. This proves (1).

Applying the moreover part in Lemma~\ref{lem:positive-branch-one-sixth} to $C=B$, we obtain
\begin{align*}
    B\simeq\mathbb P^1,\qquad \mathrm{Diff}_{B}(0) = x+\frac{1}{2}p_2+\frac{2}{3}p_3
\end{align*}
for three distinct closed points $x,p_2,p_3\in B$. By inversion of adjunction, $(X,B)$ is plt along $B\setminus\{x\}$. The local description of plt surface pairs and the coefficients of the different show that $p_2,p_3$ are cyclic quotient singularities of $X$, and that $X$ is smooth along $B\setminus\{x,p_2,p_3\}$.

Finally, $(X,\frac{6}{7}B)$ is an lc log Calabi-Yau pair satisfying Set-up 3.2 of \cite{LiuShokurov2023optimal} with $b=\frac{6}{7}$. \cite[Proposition~3.5]{LiuShokurov2023optimal} implies that $(X, \frac{6}{7}B)$ is klt. In particular, $X$ is klt. Since $(X,B)$ coincides with $(X,0)$ on $X\setminus B$, there is no non-klt center outside $B$. This proves (2).
\end{proof}

\begin{lemma}\label{lem:boundary-center-extraction}
Let $(X,B)$ be an equality pair, and let $x$ be the zero-dimensional non-klt center on the boundary. Then there is a projective birational morphism $f\colon Y\to X$ with the following properties:
\begin{enumerate}[label={\rm(\arabic*)}]
    \item $Y$ is klt;
    \item $K_Y+B_Y+E=f^*(K_X+B)$, where $B_Y:=f_*^{-1}B$ and $\Exc(f)=E$;
    \item $E$ is a smooth rational curve and meets $B_Y$ transversely at exactly one smooth point $y\in Y$.
\end{enumerate}
In particular, $(Y,B_Y+E)$ is lc and $E$ is a lc place of $(X,B)$ centered at $x$.
\end{lemma}

\begin{proof}
Recall that $B$ is smooth and that $x$ is a singular point of $X$. Let $h\colon\tilde X\to X$ be a projective partial resolution which is an isomorphism over $X\setminus\{x\}$ and is the minimal resolution over a neighborhood of $x$.

Since $B$ has only one analytic branch at $x$, the classification in \cite[Section 3, Fig 3(10), P. 58]{Kollar92Utah} implies that the extended dual graph over $x$ has the following form:
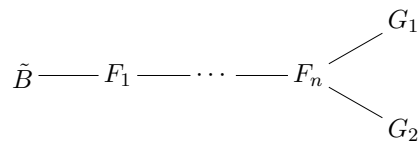
\begin{center}
\scalebox{0.9}{
\begin{tikzpicture}[
    baseline=(current bounding box.center),
    every node/.style={inner sep=2pt}
]
    \node (B)  at (0,0) {$\tilde B$};
    \node (F1) at (1.4,0) {$F_1$};
    \node (dots) at (2.8,0) {$\cdots$};
    \node (Fn) at (4.2,0) {$F_n$};
    \node (G1) at (5.6,0.8) {$G_1$};
    \node (G2) at (5.6,-0.8) {$G_2$};

    \draw (B)--(F1);
    \draw (F1)--(dots);
    \draw (dots)--(Fn);
    \draw (Fn)--(G1);
    \draw (Fn)--(G2);
\end{tikzpicture}
}
\end{center}
where $\tilde{B}:=h_*^{-1}B$ and $F_n$ meets two $(-2)$-curves $G_1$ and $G_2$. Moreover,
\begin{align}\label{eq:crepant-boundary-minimal-resolution}
    K_{\tilde X}+\tilde{D}:=K_{\tilde X}+\tilde B
    +\sum_{i=1}^{n}F_i
    +\frac{1}{2}G_1+\frac{1}{2}G_2
    =
    h^*(K_X+B).
\end{align}
For a sufficiently small $\epsilon>0$, run a $(K_{\tilde{X}}+\tilde{D}-\epsilon F_1)$-MMP over $X$. Since
\begin{align*}
    K_{\tilde X}+\tilde D-\epsilon F_1
    \equiv_X-\epsilon F_1,
\end{align*}
the MMP contracts all $h$-exceptional curves other than $F_1$ and does not contract $F_1$. Let $g\colon\tilde X\to Y$ be the composite of the contractions, which induces a projective birational morphism $f\colon Y\to X$.

The classification in \cite[Section~3, P57]{Kollar92Utah} also shows that every connected component contracted by $g$ is the resolution graph of a klt surface singularity. Since $\tilde X$ is smooth near $h^{-1}(x)$, it follows that $Y$ is klt over $x$. Since $f$ is an isomorphism away from $f^{-1}(x)$ and $X$ is klt, $Y$ is klt away from $f^{-1}(x)$. Therefore $Y$ is klt. This proves (1)

Set $E:=g_*F_1, B_Y:=g_*\tilde{B}$. Since $F_1$ is the only $h$-exceptional curve not contracted by $g$, we have $\mathrm{Exc}(f)=E$. Since every step of the MMP is crepant with respect to $K_{\tilde X}+\tilde D$, \eqref{eq:crepant-boundary-minimal-resolution} gives
\begin{align}\label{eq:crepant-boundary-partial-resolution}
    K_Y+B_Y+E=f^*(K_X+B).
\end{align}
Consequently, $E$ is an lc place of $(X,B)$ centered at $x$. This proves (2)

Since no $g$-exceptional curve passes through the point $\tilde B\cap F_1$, $g$ is an isomorphism in a neighborhood of this point. Consequently, $E$ meets $B_Y$ transversely at a smooth point of $Y$.

Since $(Y,B_Y+E)$ is lc, so is $(Y,E)$. Moreover,
\begin{align*}
    (K_Y+E)\cdot E
    &=(K_Y+B_Y+E)\cdot E-B_Y\cdot E\\
    &=-B_Y\cdot E=-1.
\end{align*}
Let $\nu\colon E^\nu\to E$ be the normalization. By adjunction,
\begin{align*}
    2g(E^\nu)-2+\deg \mathrm{Diff}_{E^\nu}(0)=(K_Y+E)\cdot E=-1.
\end{align*}
Since the $\mathrm{Diff}_{E^\nu}(0)$ is effective, $E^\nu\simeq\mathbb P^1$ and $\deg\mathrm{Diff}_{E^\nu}(0)=1$. If $E$ were not normal, its conductor would contribute at least two to the degree of the different, a contradiction. Hence $E$ is normal, and therefore $E$ is a smooth rational curve. This proves (3).
\end{proof}

\begin{proof}[Proof of Theorem \ref{thm:equality-properties}(3)]
By the proof of Theorem \ref{thm:equality-properties}(2), the pair $(X, \frac{6}{7}B)$ is klt, and Theorem \ref{thm:equality-properties}(1) gives $K_X+\frac{6}{7}B\sim_{\bQ}0$. Thus $(X,\frac{6}{7}B)$ is a projective klt log Calabi-Yau pair, and hence is exceptional.

It remains to show that $\delta(X, \frac{6}{7}B)=2$. The prime divisor $B$ itself has coefficient $\frac{6}{7}$ in the boundary, so $\delta(X, \frac{6}{7}B)\geq 1$. Let $f\colon Y\to X$ and $E$ be as in Lemma~\ref{lem:boundary-center-extraction}. We will show that $a(E, X, \frac{6}{7}B)\leq \frac{1}{7}$.

Since $E$ is a lc place of $(X, B)$,
\begin{align*}
    a(E, X, B)=a(E, X, 0)-\ord_E(B)=0
\end{align*}
gives
\begin{align*}
    \ord_E(B)=a(E, X, 0).
\end{align*}
The divisor $E$ corresponds to a component of the minimal resolution of the klt surface singularity $x\in X$. Hence $a(E,X,0)\leq 1$. Since $E$ is a lc place of $(X, B)$,
\begin{align*}
    a(E, X, \frac{6}{7}B)&=a(E, X, B)+\frac{1}{7}\ord_E(B)\\
    &=\frac{1}{7}\ord_E(B)\\
    &=\frac{1}{7}a(E,X,0)\leq \frac{1}{7}.
\end{align*}
The valuations defined by $B$ and $E$ are distinct, so $\delta(X, \frac{6}{7}B)\geq 2$. By Theorem \ref{thm:delta-leq-2}, $\delta(X, \frac{6}{7}B)\leq 2$. Hence $\delta(X, \frac{6}{7}B)=2$.
\end{proof}

\section{Uniqueness of the equality pair}
\label{sec:uniqueness}

We prove that the equality pair constructed in Proposition~\ref{prop:construction} is unique up to isomorphism.

\begin{theorem}\label{thm:uniqueness}
Let $(X,B)$ be an equality pair, then $(X,B)\simeq(X_{18},B_{18})$, where
\begin{align}\label{eq:X18-definition}
    X_{18}:=\left\{z^2+y^2(y-x^3)+x^2t^2=0\right\}
    \subset\mathbb P(2_x,6_y,7_t,9_z),\qquad
    B_{18}:=(t=0)|_{X_{18}}.
\end{align}
\end{theorem}

We reduce an equality pair to one of two rank-one pairs as in Theorem \ref{thm:rank-one-exceptional} and determine the possible extractions. We then use explicit sections to identify the ample model and its boundary.

\subsection{The model \texorpdfstring{$(X_{18}, B_{18})$}{(X18, B18)}}

\begin{proposition}\label{prop:X18-structure}
    The hypersurface $X_{18}$ defined in Theorem \ref{thm:uniqueness} is a normal projective surface and $B_{18}$ is a smooth rational curve. Moreover, 
\begin{align*}
    (K_{X_{18}}+B_{18})^2&=\frac{1}{42}.
\end{align*}
\end{proposition}

\begin{proof}
Write
\begin{align*}
    X_{18}=(F=0)\subseteq \bP(2,6,7,9),
    \qquad
    F=z^2+y^2(y-x^3)+x^2t^2,
\end{align*}
where $\mathrm{wt} (x,y,t,z)=(2,6,7,9)$.

A direct computation shows that the singular locus of the affine cone $(F=0)\subset\bA^4$ is $\{y=t=z=0\}\cup\{x=y=z=0\}$. This set has codimension two in the cone. Therefore $(F=0)\subseteq \bA^4$ is regular in codimension one. Since a hypersurface is Cohen--Macaulay, it satisfies $S_2$. Hence the cone is normal by Serre's criterion. Note that $F$ is irreducible. Thus $R:=\bC[x,y,t,z]/(F)$ is a integrally closed domain. With respect to the weighted grading, the degree zero part of $R$ is $R_0=\bC$. Hence $R$ satisfies the
hypotheses of \cite[Lemma~5.23]{Miyanishi94}, and thus $X_{18}=\operatorname{Proj}R$ is normal.

    We then show that $B_{18}$ is smooth. Using the isomorphism $\operatorname{Proj}A\simeq \operatorname{Proj}A^{(d)}$ for the $d$-th Veronese subring $A^{(d)}$, first with $d=3$ and then with $d=2$, and relabeling the resulting homogeneous generators, we obtain
    \begin{align*}
        B_{18}=&(z^2+y^2(y-x^3)=0) \subseteq \bP(2,6,9)\\
        \simeq &(z^2+y^2(y-x)=0) \subseteq \bP(2,2,3)\\
        \simeq &(z+y^2(y-x)=0) \subseteq \bP(1,1,3).
    \end{align*}
    Thus $B_{18}\simeq~\mathrm{Proj} \bC[x,y,z]/(z+y^2(y-x))$ with $\mathrm{wt}(x,y,z)=(1,1,3)$. Since
    \begin{align*}
        \bC[x,y,z]/(z+y^2(y-x)) \simeq \bC[x,y,y^2(y-x)]\simeq \bC[x,y],
    \end{align*}
    we have $B_{18}\simeq \bP^1$.

    Since $\cO_{\bP(2,6,7,9)}(1)$ is $\bQ$-Cartier and
    \begin{align*}
        \cO_{X_{18}}(K_{X_{18}}+B_{18})&\sim_\bQ \cO_{X_{18}}(1),
    \end{align*}
the remaining statements follow from

\begin{align*}
    \cO_{X_{18}}(1)^2=\frac{18}{2\cdot 6\cdot 7\cdot 9}=\frac{1}{42}.
\end{align*}
\end{proof}

\subsection{A birational morphism to a rank-one model}

\begin{lemma}\label{lem:rank-one-reduction}
Let $(X,B)$ be an equality pair. Let $f\colon Y\to X$ be the morphism in Lemma \ref{lem:boundary-center-extraction} and put $m:=\mathrm{ord}_E(B)$. Then $m=1$, and there are a rank-one surface $S$, two curves $C_1,C_2\subset S$, and a projective birational morphism $h\colon Y\to S$ such that $(S,\frac{6}{7}(C_1+C_2))$ is of type $(A_2^6)$ or $(I_2^2)$ and
\begin{align*}
    K_Y+\frac{6}{7}(B_Y+E)
    =h^*\left(K_S+\frac{6}{7}(C_1+C_2)\right).
\end{align*}
\end{lemma}

\begin{proof}
By the proof of Theorem~\ref{thm:equality-properties}{\rm(3)}, we have
\begin{align}\label{eq:m-range}
    0<m=\ord_E(B)=a(E,X,0)\leq 1.
\end{align}

Set $\Delta_Y:=\frac{6}{7}B_Y+\left(1-\frac{m}{7}\right)E$, then $\Delta_Y\in \Phi_{\mathrm{m}}$. Subtracting $\frac{1}{7}f^*B$ from both sides of
\begin{align*}
    K_Y+B_Y+E=f^*(K_X+B),
\end{align*}
we obtain
\begin{align}\label{eq:rank-one-crepant-X}
    K_Y+\Delta_Y
    =
    f^*\left(K_X+\frac{6}{7}B\right)
    \sim_{\bQ}0.
\end{align}
By Theorem~\ref{thm:equality-properties}~{\rm(3)} and
\eqref{eq:rank-one-crepant-X}, the pair $(Y,\Delta_Y)$ is exceptional and $\delta(Y,\Delta_Y)=2$.

Run a $(K_Y+\Delta_Y-\epsilon E)$-MMP for a sufficiently small $\epsilon>0$, which terminates with a Mori fiber space $g\colon S\to Z$. Denote the induced morphism by $h\colon Y\to S$ and set $\Delta_S:=h_*\Delta_Y$. Since $K_Y+\Delta_Y\sim_{\bQ}0$,
\begin{align}\label{eq:rank-one-crepant-reduction}
    K_Y+\Delta_Y=h^*(K_S+\Delta_S).
\end{align}
Consequently, $(S,\Delta_S)$ is an exceptional log Calabi-Yau pair and $\delta(S,\Delta_S)=2$.

We next show that neither $E$ nor $B_Y$ is contracted by $h$. Since this MMP is also a $(-E)$-MMP, it does not contract $E$. Set $D_S:=h_*(B_Y+E)$. Since $(S,\Delta_S)$ is klt and $\Delta_S\geq\frac{6}{7}D_S$, Lemma \ref{lem:surface-one-gap} implies that $(S, D_S)$ is lc. Since $E$ is the unique $f$-exceptional curve and is not contracted by $h$, for any $h$-exceptional curve $C$, we have
\begin{align*}
    (K_Y+B_Y+E)\cdot C=(K_X+B)\cdot f_*C>0.
\end{align*}
Hence $K_Y+B_Y+E$ is $h$-ample. As both $(Y,B_Y+E)$ and $(S,D_S)$ are lc, the negativity lemma shows that $B_Y$ is not contracted by $h$.

We may now set $C_1:=h_*B_Y,C_2:=h_*E$ and write $\Delta_S=\frac{6}{7}C_1+(1-\frac{m}{7})C_2$.

We claim that $Z$ is a point. Since
\begin{align*}
    K_S+\Delta_S-\epsilon C_2
    \equiv-\epsilon C_2
\end{align*}
is anti-ample over $Z$, the curve $C_2$ is ample over $Z$. In particular, it is horizontal over $Z$ if $Z$ is a curve. Applying \cite[Lemma~3.4]{LiuShokurov2023optimal} with
\begin{align*}
    b:=\min\{\frac{6}{7},1-\frac{m}{7}\}
\end{align*}
would give
\begin{align*}
    b\leq\frac{2}{3},
\end{align*}
contradiction. Hence $Z$ is a point and $\rho(S)=1$.

The two components $C_1$ and $C_2$ have log discrepancies $\frac{1}{7}$ and $\frac{m}{7}$ in $\Delta_S$, both at most $\frac{1}{7}$. Since $\delta(S,\Delta_S)=2$, every divisor centered at a closed point of $S$ has log discrepancy greater than $\frac{1}{7}$. Therefore $(S,\Delta_S)$ is $\frac{1}{7}$-lt at closed points.

We can now apply the rank-one classification in Theorem \ref{thm:rank-one-exceptional}. Since the coefficients of $\Delta_S$ are at least $\frac{6}{7}$, the classification leaves precisely the two cases
\begin{align*}
    (A_2^6)
    \qquad\text{and}\qquad
    (I_2^2).
\end{align*}
In both cases, $\Delta_S=\frac{6}{7}(C_1+C_2)$. It follows that $m=1$ and 
\begin{align*}
    K_Y+\frac{6}{7}(B_Y+E)=h^*\left(K_S+\frac{6}{7}(C_1+C_2)\right).
\end{align*}
\end{proof}

\subsection{Local models and coordinates}

For each of the two rank-one pairs, we add auxiliary curves to obtain a toroidal pair. We then choose local coordinates for the computations below.

For $S=\mathbb P(1,2,3)_{[x_1:x_2:x_3]}$, the two singular weighted affine charts are
\begin{align*}
    U_2:=(x_2\neq0)
    \simeq
    \mathbb A^2_{x_1,x_3}/\mu_2(1,1),\qquad 
    U_3:=(x_3\neq0)
    \simeq
    \mathbb A^2_{x_1,x_2}/\mu_3(1,2).
\end{align*}
Here, by abuse of notation, we use the same symbols $x_i$ for the weighted homogeneous coordinates on $S$ and their pullbacks to the corresponding orbifold covers. Equations on $U_2$ and $U_3$ are always understood after dehomogenizing by $x_2=1$ and $x_3=1$ on the respective orbifold covers.

\begin{lemma}\label{lem:toroidal-completion}
Let $(S,\frac{6}{7}(C_1+C_2))$ be one of the rank-one exceptional pairs of type $(A_2^6)$ or $(I_2^2)$. Let $Q_2$ and $Q_3$ be the cyclic quotient singularities of types $\frac{1}{2}(1,1)$ and $\frac{1}{3}(1,2)$ on $S$, respectively. Then we may choose general members
\begin{align*}
    C_{Q_2}\in |\mathcal O_S(3)|,
    \qquad
    C_{Q_3}\in |\mathcal O_S(2)|
\end{align*}
through $Q_2$ and $Q_3$, respectively, such that $(S, D)$ is a toroidal pair, where $D:=C_1+C_2+C_{Q_2}+C_{Q_3}$.
\end{lemma}

\begin{proof}
Choose weighted homogeneous coordinates $S=\bP(1,2,3)_{[x_1:x_2:x_3]}$ as in Lemma~\ref{lem:rank-one-normal-forms}. Thus
\begin{align*}
    Q_2=[0:1:0],
    \qquad
    Q_3=[0:0:1].
\end{align*}

We also write $C_{(k)}$ for the degree-$k$ curve in the boundary. Since
\begin{align*}
    H^0(S,\mathcal O_S(3))=\mathrm{Span}_\bC
    \{x_1^3,x_1x_2,x_3\}, \qquad
    H^0(S,\mathcal O_S(2))=\mathrm{Span}_\bC
    \{x_1^2,x_2\},
\end{align*}
we have
\begin{align*}
    \mathrm{Bs}|\cO_{S}(3)|=\{Q_2\},\qquad 
    \mathrm{Bs}|\cO_{S}(2)|=\{Q_3\}.
\end{align*}
Choose
\begin{align*}
    C_{Q_2}:=
    \left(
        x_3+\alpha x_1x_2+\beta x_1^3=0
    \right), \qquad
    C_{Q_3}:=
    \left(
        x_2+\gamma x_1^2=0
    \right),
\end{align*}
where $\alpha\neq0$ and $\alpha,\beta,\gamma$ are otherwise general. Then $C_{Q_2}$ contains $Q_2$ but not $Q_3$, while $C_{Q_3}$ contains $Q_3$ but not $Q_2$.

We first show that $(S, D)$ is toroidal at $Q_2$ and $Q_3$. On the orbifold chart $U_2$, the pullback of $C_{Q_2}$ is defined by
\begin{align*}
    C_{Q_2}=\left(x_3+\alpha x_1+\beta x_1^3=0\right).
\end{align*}
In the $(A_2^6)$ case,  the unique component of $C_1+C_2$ through $Q_2$ is $C_{(1)}=(x_1=0)$, and $(x_1,x_3+\alpha x_1+\beta x_1^3)$ is a $\mu_2$-equivariant regular system of parameters at the origin of the orbifold cover. In the $(I_2^2)$ case, the corresponding component is $C_{(3)}=(x_3=0)$, and $(x_3,x_3+\alpha x_1+\beta x_1^3)$ is a $\mu_2$-equivariant regular system of parameters at the origin of the orbifold cover, because $\alpha\neq0$.

On the orbifold chart $U_3$, the pullback of $C_{Q_3}$ is defined by
\begin{align*}
    C_{Q_3}=\left(x_2+\gamma x_1^2=0\right).
\end{align*}
In the $(A_2^6)$ case, the unique component of $C_1+C_2$ through $Q_3$ is $C_{(1)}=(x_1=0)$, and $(x_1, x_2+\gamma x_1^2)$ is a $\mu_3$-equivariant regular system of parameters at
the origin of the orbifold cover. In the $(I_2^2)$ case, the corresponding component is $C_{(4)}=(x_2^2-x_1^4-x_1=0)$. Since $x_2^2-x_1^4-x_1$ has linear term $-x_1$ on the orbifold cover, the ordered pair $(x_2^2-x_1^4-x_1,x_2+\gamma x_1^2)$ is a $\mu_3$-equivariant regular system of parameters at the origin of the orbifold cover.

Thus, in each case, the displayed pair defines a $\mu_r$-equivariant \'etale coordinate system on the orbifold cover, with the two branches of $D$ corresponding to the coordinate axes. Hence it induces the required local toric model of $(S,D)$ at $Q_2$ or $Q_3$, respectively.

Away from $Q_2$ and $Q_3$, both linear systems are base-point-free. By Bertini's theorem and the generality of $\alpha,\beta,\gamma$, the curves $C_{Q_2}$ and $C_{Q_3}$ are smooth there, meet each other and $C_1+C_2$ transversely at smooth points of $S$, and avoid the zero-dimensional strata of $C_1+C_2$. We may also assume that no point lies on three analytic branches of $D$. Together with the preceding local calculations at $Q_2$ and $Q_3$, $(S, D)$ is toroidal at all zero-dimensional strata of $(S, D)$. Since $S\setminus D$ is smooth, $(S, D)$ is a toroidal pair.
\end{proof}

\begin{remark}[Local models at cyclic quotient singularities]\label{rem:local-model-singularities}
    Fix the boundary $D$ and the local models constructed in the proof of Lemma~\ref{lem:toroidal-completion}.

    On the orbifold chart $x_2=1$, put $v:=x_3+\alpha x_1+\beta x_1^3=0$. The pairs of ordered orbifold parameters are $(x_1,v)$ in the $(A_2^6)$ case and $(x_3,v)$ in the $(I_2^2)$ case. The primitive generators of the corresponding rays in the local model are $(2,-1)$ and $(0,1)$ in this order.

    On the orbifold chart $x_3=1$, put $v:=x_2+\gamma x_1^2$. The pairs of ordered orbifold parameters are $(x_1,v)$ in the $(A_2^6)$ case and $(G_4,v)$ in the $(I_2^2)$ case. The primitive generators of the corresponding rays in the local model are $(3,-2)$ and $(0,1)$ in this order.

    In both cases, the first ray corresponds to the branch of $C_1+C_2$, and the second to the branch of the auxiliary curve $C_{Q_i}$.
\end{remark}

We next record analytic coordinates at the node of $C_{(6)}$ in the $(A_2^6)$ case. Choose weighted coordinates on $S=\bP(1,2,3)$ as in Lemma~\ref{lem:rank-one-normal-forms}, and set 
\begin{align*}
    u:=x_2-x_1^2,
    \qquad
    G_6:=x_3^2-3x_1^2u^2-u^3.
\end{align*}
Then
\begin{align*}
    C_{(1)}=(x_1=0),
    \qquad
    C_{(6)}=(G_6=0).
\end{align*}
The curve $C_{(6)}$ has a node at $P=[1:1:0]$.

\begin{lemma}\label{lem:A26-local-coordinates}
There are analytic coordinates $(r,w)$ near $P$ and an analytic unit $\theta$ defined near $P$ such that
\begin{align*}
    G_6=rw,
    \qquad
    u=\frac{w-r}{2\theta},
    \qquad
    x_3=\frac{r+w}{2}.
\end{align*}
In particular, the two analytic branches of $C_{(6)}$ at $P$ are defined by $r=0$ and $w=0$.
\end{lemma}

\begin{proof}
    Consider the affine chart $U_1=(x_1\neq0)\subset S$, set $x_1=1$ and retain the notation $x_2,x_3$ for the affine coordinates.  The point $P$ is then given by $(u,x_3)=(0,0)$ and $(u,x_3)$ are analytic coordinates $P$. Moreover, on this chart,
\begin{align*}
    u=x_2-1, \qquad G_6=x_3^2-u^2(3+u).
\end{align*}
Since $(3+u)(P)=3\neq0$, after shrinking, choose an analytic unit $\theta$ such that $\theta^2=3+u$. Set
\begin{align*}
    r:=x_3-\theta u, \qquad w:=x_3+\theta u.
\end{align*}
Since
\begin{align*}
    \det\frac{\partial(r,w)}{\partial(u,x_3)}(0,0)=-2\theta(P)\neq 0,
\end{align*}
the analytic inverse function theorem shows that $(r,w)$ are analytic coordinates at $P$. Finally,
\begin{align*}
    G_6=rw,\qquad
    u=\frac{w-r}{2\theta},\qquad
    x_3=\frac{r+w}{2}.
\end{align*}
\end{proof}

We next record analytic coordinates at a crossing of $C_{(3)}$ and $C_{(4)}$ in the $(I_2^2)$ case. Choose weighted coordinates on $S=\bP(1,2,3)$ as in Lemma~\ref{lem:rank-one-normal-forms}, and set
\begin{align*}
    u:=x_2-x_1^2,
    \qquad
    G_4:=x_2^2-x_1^4-x_1x_3.
\end{align*}
Then
\begin{align*}
    C_{(3)}=(x_3=0),
    \qquad
    C_{(4)}=(G_4=0).
\end{align*}
The curves $C_{(3)}$ and $C_{(4)}$ meet transversely at
\begin{align*}
    P_+=[1:1:0],
    \qquad
    P_-=[1:-1:0].
\end{align*}
We consider $P=P_+$, as the point $P_-$ is obtained from $P_+$ by the involution $x_2\mapsto -x_2$, which preserves both curves.

\begin{lemma}\label{lem:I22-local-coordinates}
There are analytic coordinates $(r,w)$ at $P$ such that
\begin{align*}
    x_3=(2+r+w)r,
    \qquad
    G_4=(2+r+w)w,
    \qquad
    u=r+w.
\end{align*}
In particular,
\begin{align*}
    C_{(3)}=(r=0),
    \qquad
    C_{(4)}=(w=0)
\end{align*}
in a neighborhood of $P$.
\end{lemma}

\begin{proof}
Consider the affine chart $U_1=(x_1\neq0)\subset S$, set $x_1=1$ and retain the notation $x_2,x_3$ for the affine coordinates. The point $P$ is then given by $(u,x_3)=(0,0)$ and $(u,x_3)$ are analytic coordinates $P$. Moreover, on this chart, 
\begin{align*}
    u=x_2-1, \qquad G_4=x_2^2-1-x_3=u(2+u)-x_3.
\end{align*}
Since $(2+u)(P)=2\neq 0$, after shrinking, the function $2+u$ is a unit near $P$. Set
\begin{align*}
    r:=\frac{x_3}{2+u},
    \qquad
    w:=\frac{G_4}{2+u}.
\end{align*}
Since $(x_3,G_4)$ are analytic coordinates at $P$, so are $(r,w)$. Moreover,
\begin{align*}
    x_3=(2+r+w)r,
    \qquad
    G_4=(2+r+w)w,
    \qquad
    u=r+w.
\end{align*}
In particular, $C_{(3)}=(r=0)$ and $C_{(4)}=(w=0)$ in this neighbourhood.
\end{proof}

We finally record the order of vanishing over $Q_3$ needed below. 

\begin{corollary}\label{cor:F7-order-u}
Let $F_7$ be the prime divisor over $S$ determined by the primitive vector $(7,-4)$ in the fixed local model at $Q_3$ from Remark~\ref{rem:local-model-singularities}. Then for $u=x_2-x_1^2$,
we have
\begin{align*}
    \ord_{F_7}(u)=\frac{2}{3}.
\end{align*}
\end{corollary}

\begin{proof}
By Lemma~\ref{lem:toric-orders}, the $F_7$-orders of the ordered orbifold parameters $(x_1,v)$ or $(G_4,v)$ are $(\frac{7}{3}, \frac{2}{3})$. In particular, $\ord_{F_7}(v)=\frac{2}{3}$.

On the other hand, $u-v=-(1+\gamma)x_1^2$. In either of the orbifold parameters, $x_1$ belongs to the maximal ideal $(x_1,v)$ or $(G_4,v)$ at the origin of the orbifold cover. Hence every nonzero term of $x_1^2$ has degree at least two, and therefore has $F_7$-order at least $\frac{4}{3}$. Thus $\ord_{F_7}(u-v)>\ord_{F_7}(v)$, and hence
\begin{align*}
    \ord_{F_7}(u)
    =
    \min\bigl\{\ord_{F_7}(v),\ord_{F_7}(u-v)\bigr\}
    =
    \frac{2}{3}.
\end{align*}
If $u=v$, this is immediate.
\end{proof}

\subsection{Classification of the crepant extractions}

By Lemma \ref{lem:rank-one-reduction}, it is sufficient to study $h$ when $(S, \Delta)$ is one of the rank-one exceptional pairs of type $(A_2^6)$ or $(I_2^2)$. In this section, we use the chosen toroidal completion $(S,D)$ as in Lemma \ref{lem:toroidal-completion} to determine the possible configurations of the exceptional curves of $h$ in the corresponding local models.

\begin{lemma}\label{lem:h-toroidal}
Let $(S, \Delta)$ be one of the rank one exceptional pairs of type $(A_2^6)$ or $(I_2^2)$, where $\Delta:=\frac{6}{7}(C_1+C_2)$. Let $D$ be as defined in Lemma \ref{lem:toroidal-completion}. Let $h\colon Y\to S$ be a projective birational morphism between normal surfaces such that $K_Y+h_*^{-1}\Delta=h^*(K_S+\Delta)$. Then $(Y, D_Y)$ is a toroidal pair, where $D_Y:=h_*^{-1}D+\Exc(h)$. Moreover, $h\colon (Y, D_Y)\to (S, D)$ is a toroidal morphism.
\end{lemma}

\begin{proof}
By Theorem \ref{thm:rank-one-exceptional}, all singularities of $S$ are contained in $\Supp(\Delta)$, thus $S\setminus \Supp(\Delta)$. Theorem \ref{thm:rank-one-exceptional} also gives that $(S, \Delta)$ is klt. By definition of $D$, $\Supp(\Delta)\subseteq \Supp(D)$. Hence all assumptions of Lemma \ref{lem:crepant-toroidal} are satisfied and the lemma follows.
\end{proof}

\begin{lemma}\label{lem:exhaustion}
Let $(S,\Delta)$ be one of the rank-one exceptional pairs of type $(A_2^6)$ or $(I_2^2)$, where $\Delta:=\frac{6}{7}(C_1+C_2)$. Denote $C_{(k)}$ for the component $C_1+C_2$ of degree $k$. In the $(A_2^6)$ case, we further assume that $C_{(6)}$ is an irreducible nodal rational curve.

Let $h\colon Y\to S$ be a projective birational morphism from a normal surface, $B_Y,E$ are irreducible curves on $Y$ such that $B_Y+E=h_*^{-1}(C_1+C_2)$. We write
\begin{align*}
    D_Y:=h_*^{-1}D+\Exc(h),\qquad L:=K_Y+B_Y+E
\end{align*}
Assume that
\begin{enumerate}[label={\rm(\arabic*)}]
    \item $h\colon (Y, D_Y)\to (S, D)$ is a toroidal morphism and $K_Y+\Delta_Y=h^*\left(K_S+\Delta\right)$, where $\Delta_Y:=\frac{6}{7}(B_Y+E)$;
    \item $L$ is big and nef, and $L^2=\frac{1}{42}$;
    \item there is a unique irreducible curve on $Y$ having zero intersection with $L$; this curve is one of $B_Y$ and $E$, while the other has intersection $\frac{1}{6}$ with $L$;
    \item $B_Y$ and $E$ intersect transversely at a smooth point of $Y$, and among the smooth points of $S$ at which $C_1+C_2$ has two analytic branches, there is exactly one point $P$ over which $h$ is not an isomorphism.
\end{enumerate}
Then $h$ is an isomorphism over $S\setminus\{P\}$. The corresponding fan subdivision in the local model at $P$ is given by
\begin{center}
\renewcommand{\arraystretch}{1.18}
\begin{tabular}{@{}cl@{}}
\toprule
rank-one pair
&
extraction over the point with two analytic branches
\\
\midrule
$(A_2^6)$
&
$(5,2),(2,5)$ at the node of $C_{(6)}$
\\
$(I_2^2)$
&
$(5,2),(2,5)$ at one of the two crossings of
$C_{(3)}+C_{(4)}$
\\
\bottomrule
\end{tabular}
\end{center}

\end{lemma}

\begin{proof}

Let $G$ be a prime $h$-exceptional curve. Condition (1) gives $a(G, S, \Delta)=1$. On the other hand,
\begin{align*}
    K_Y+\frac{6}{7}(B_Y+E)=h^*(K_S+\frac{6}{7}(C_1+C_2))\sim_\bQ 0.
\end{align*}
Thus $L\equiv \frac{1}{7}(B_Y+E)$ and Condition~{\rm(3)} gives
\begin{align}\label{eq:exceptional-meets-boundary}
    (B_Y+E)\cdot G=7L\cdot G>0.
\end{align}
Thus every $h$-exceptional curve meets $B_Y+E$.

\medskip
\noindent\emph{Step 1. Determine all possible configurations of $h$-exceptional curves}. 

By Condition (1) and Definition \ref{def:toroidal-local-model}, $h(\Exc(h))\subseteq D$.

Suppose first that an $h$-exceptional curve is centered at a smooth point of $S$ lying on two analytic branches of $C_1+C_2$. These branches meet transversely. At such a point, choose a local model $(X_\sigma,t)$ such that $\sigma$ has primitive generators $u_0=(1,0)$ and $u_{n+1}=(0,1)$, corresponding to the two branches of $C_1+C_2$. 

Let $u_i=(p_i,q_i)$, $1\leq i\leq n$, be the primitive generators of the new rays in the subdivision of $\sigma$ induced by $h$, with corresponding exceptional curves $E_i$. Lemma~\ref{lem:local-toric-discrepancy} gives
\begin{align*}
    1=a\left(E_{u_i},S, \Delta\right)=\frac{p_i+q_i}{7}.
\end{align*}
Hence
\begin{align}\label{eq:possible-crossing-rays}
    (p_i,q_i) \in \bigl\{ (6,1),(5,2),(4,3),(3,4),(2,5),(1,6) \bigr\}.
\end{align}
By \eqref{eq:exceptional-meets-boundary}, every $h$-exceptional curve meets $B_Y+E$. Since only those corresponding to new rays adjacent to $\bR_{\geq 0}u_0$ or $\bR_{\geq 0}u_{n+1}$ can do so, $n\leq2$. Let $P$ be the unique smooth two-branch point over which $h$ is not an isomorphism, as in Condition~{\rm(4)}.

Suppose next that an $h$-exceptional curve is centered at a smooth point lying on exactly one branch of $C_1+C_2$. Since $h$ is toroidal, such a center must be a zero-dimensional stratum of $D$. At such a point, choose a local model $(X_\sigma, t)$ such that $\sigma$ has primitive generators $(1,0)$ and $(0,1)$, corresponding respectively to the branch of $C_1+C_2$ and the other component of $D$. The latter has coefficient zero in $\Delta$.

Let $u=(p,q)$ be the primitive generator of a new ray in the subdivision of $\sigma$ induced by $h$, with corresponding exceptional curves $E$. Lemma~\ref{lem:local-toric-discrepancy} gives
\begin{align*}
    a(E_u,S,\Delta)=\frac{p}{7}+q>1,
\end{align*}
contrary to Condition~{\rm(1)}. Hence no $h$-exceptional curve is centered at such a point.

It remains to consider exceptional curves centered at the two cyclic quotient singularities $Q_2$ and $Q_3$. At $Q_2$, we use the fixed local model of Remark \ref{rem:local-model-singularities}. The primitive generators $(2,-1), (0,1)$ correspond to the branches of $C_1+C_2$, $C_{Q_2}$, respectively. The latter branch has coefficient zero in $\Delta$.

Let $u=(p,q)$ be the primitive generator of a new ray in the subdivision of $\sigma$ induced by $h$, with corresponding exceptional curves $E$. Condition (1) and Lemma~\ref{lem:local-toric-discrepancy} give
\begin{align*}
    1=a\left(E,S, \Delta \right)=\frac{4p+7q}{7}.
\end{align*}
The only primitive solution is $(7,-3)$. We denote the corresponding curve over $S$ by $F_{(7,-3)}$.

At $Q_3$, we use the fixed local model of Remark \ref{rem:local-model-singularities}. The primitive generators $(3,-2), (0,1)$ correspond to the branches of $C_1+C_2$, $C_{Q_3}$, respectively. The latter branch has coefficient zero in $\Delta$.

Let $u=(p,q)$ be the primitive generator of a new ray in the subdivision of $\sigma$ induced by $h$, with corresponding exceptional curves $E$. Condition (1) and Lemma~\ref{lem:local-toric-discrepancy} give
\begin{align*}
    1=a\left(E,S, \Delta\right)=\frac{5p+7q}{7}.
\end{align*}
The only primitive solutions are $(7,-4)$ and $(14,-9)$. We denote their corresponding curves over $S$ by $F_7$ and $F_{14}$, respectively. 

At most one of $F_7$ and $F_{14}$ can be extracted by $h$. Otherwise, the ray of $F_7$ would not be adjacent to $\bR_{\geq0}(3,-2)$, so $F_7$ would not meet $B_Y+E$, contradicting \eqref{eq:exceptional-meets-boundary}.

For later use, set
\begin{align*}
    R_h:=h^*(K_S+C_1+C_2)-L.
\end{align*}
Then $R_h$ is $h$-exceptional and for every prime $h$-exceptional curve $F$,
\begin{align*}
    \mathrm{mult}_F(R_h)=1-a(F,S,C_1+C_2).
\end{align*}
For the candidate exceptional curves identified above, their log discrepancies with respect to the pair $(S,C_1+C_2)$ are
\begin{align*}
    a(F_{(7,-3)},S,C_1+C_2)=\frac{1}{2},\quad
    a(F_7,S,C_1+C_2)=\frac{2}{3},\quad
    a(F_{14},S,C_1+C_2)=\frac{1}{3},
\end{align*}
whereas $a(F,S,C_1+C_2)=0$ for every divisor $F$ extracted over a smooth two-branch point. Consequently, $R_h$ is effective, and its coefficients along $F_{(7,-3)}$, $F_7$, $F_{14}$, and a divisor extracted over a smooth two-branch point are $\frac{1}{2},\frac{1}{3},\frac{2}{3},1$, respectively.

\medskip
\noindent\emph{Step 2. Rule out the extraction over $Q_2$.}

\begin{claim}\label{claim:no-index-two-extraction}
    The curve $F_{(7,-3)}$ is not extracted by $h$.
\end{claim}

\begin{proof}
Suppose first that $(S,\frac{6}{7}(C_1+C_2))$ is of type $(A_2^6)$. The center of $F_{(7,-3)}$ lies on $C_{(1)}$, and
\begin{align*}
    (K_S+C_{(1)}+C_{(6)})\cdot C_{(1)}
    =
    \frac{1}{6}.
\end{align*}
Suppose that $F_{(7,-3)}$ is extracted. Set $C_{(1),Y}:=h_*^{-1}C_{(1)}$. Recall that in the local model, $C_{(1), Y}, F_{(7,-3)}$ correspond to primitive generators $(2,-1), (7,-3)$, respectively. Lemma \ref{lem:toric-intersections} gives 
\begin{align*}
    C_{(1), Y}\cdot F_{(7,-3)}=1
\end{align*}
and we would have
\begin{align*}
    L\cdot C_{(1),Y}&=\big(h^*(K_S+C_1+C_2)-R_h  \big)\cdot C_{(1), Y}\\
    &\leq \big(h^*(K_S+C_1+C_2)-\frac{1}{2}F_{(7,-3)}  \big)\cdot C_{(1), Y}\\
    &=\frac{1}{6}-\frac{1}{2}<0,
\end{align*}
contradicting nefness of $L$.

Suppose now that the pair is of type $(I_2^2)$. In this case, the center of $F_{(7,-3)}$ lies on $C_{(3)}$, and
\begin{align*}
    (K_S+C_{(3)}+C_{(4)})\cdot C_{(3)}=\frac{1}{2}.
\end{align*}
Recall that in the local model $(X_\sigma,t)$ at $Q_2$, $C_{(3),Y}$ and $F_{(7,-3)}$ correspond to primitive generators $(2,-1)$ and $(7,-3)$, respectively. By Condition~{\rm(4)}, $h$ is not an isomorphism over exactly one of the two points of $C_{(3)}\cap C_{(4)}$. In the local model $(X_{\sigma'},t')$ at that crossing, let $u_1=(p_1,q_1)$ generate the new ray adjacent to $\bR_{\geq0}(1,0)$,  which is the ray corresponding to $C_{(3)}$. Lemma \ref{lem:toric-intersections} gives 
\begin{align*}
    C_{(3),Y}\cdot F_{(7,-3)}&=1,\\
    C_{(3),Y}\cdot E_{1}&=\frac{1}{q_1}.
\end{align*}
Consequently
\begin{align*}
    L\cdot C_{(3),Y}&=\big(h^*(K_S+C_1+C_2)-R_h \big)\cdot C_{(3), Y}\\
    &\leq \big(h^*(K_S+C_1+C_2)-\frac{1}{2}F_{(7,-3)}-E_{1}\big)\cdot C_{(3), Y}\\
    &= \frac{1}{2} -\frac{1}{2} -\frac{1}{q_1}<0,
\end{align*}
again contradicting nefness of $L$.
\end{proof}

\medskip
\noindent\emph{Step 3. Use the volume drop to rule out the extraction of $F_{14}$.}

Define the \emph{volume drop of $h$} by
\begin{align*}
    d(h)&:=(K_S+C_1+C_2)^2-L^2.
\end{align*}
For both rank-one pairs, $(K_S+C_1+C_2)^2=\frac{1}{6}$. Hence, by Condition~{\rm(2)},
\begin{align}\label{eq:exhaustion-total-decrease}
    d(h)=\frac{1}{6}-\frac{1}{42}=\frac{1}{7}.
\end{align}

For each closed point $s\in S$, let $R_{h,s}$ be the sum of the components of $R_h$ centered at $s$, with their coefficients in $R_h$. Thus 
\begin{align*}
    R_h=\sum_s R_{h,s}, \qquad d_s(h):=-R_{h,s}^2.
\end{align*}
Since $R_h$ is $h$-exceptional, the projection formula gives $d(h)=-R_h^2$. Moreover, the divisors $R_{h,s}$ corresponding to distinct $s$ have disjoint supports. Therefore
\begin{align}\label{eq:additivity-of-volume-drop}
    \sum_s d_s(h)=\frac{1}{7}.
\end{align}

Let $P$ be the unique smooth two-branch point over which $h$ is not an isomorphism. By Lemma~\ref{lem:toric-intersections} and
\eqref{eq:additivity-of-volume-drop}, we obtain
\begin{center}
\renewcommand{\arraystretch}{1.15}
\begin{tabular}{@{}ccc@{}}
\toprule
extraction over $Q_3$
&
$d_{Q_3}(h)$
&
$d_P(h)$
\\
\midrule
none     & $0$             & $\frac{1}{7}$   \\
$F_7$    & $\frac{1}{42}$  & $\frac{5}{42}$ \\
$F_{14}$ & $\frac{2}{21}$  & $\frac{1}{21}$ \\
\bottomrule
\end{tabular}
\end{center}

If $h$ extracts a single divisor over $P$, with corresponding primitive generator $(p,q)$, then
\begin{align*}
    d_P(h)=\frac{1}{pq}.
\end{align*}
None of the vectors in \eqref{eq:possible-crossing-rays} gives any of the values in the last column of the table above. By Condition~{\rm(4)} and the bound $n\leq2$ proved above, $h$ therefore extracts exactly two curves $F_1,F_2$ over $P$. Let $(p_1,q_1)$ and $(p_2,q_2)$ be the primitive generators of the corresponding rays of $F_1$ and $F_2$, ordered by increasing slope. By Lemma~\ref{lem:toric-intersections},
\begin{align*}
    d_P(h)&=-(F_1+F_2)^2=\big(\frac{p_1}{p_2}+\frac{q_2}{q_1}-2 \big)\cdot\frac{1}{p_1q_2-p_2q_1}.
\end{align*}
Checking all pairs of vectors in \eqref{eq:possible-crossing-rays} gives the following table:
\begin{center}
\renewcommand{\arraystretch}{1.18}
\begin{tabular}{@{}ccc@{}}
\toprule
extraction over $Q_3$
&
required value of $d_P(h)$
&
ordered generators
\\
\midrule
none
&
$\frac{1}{7}$
&
$(5,2),(2,5)$
\\
$F_7$
&
$\frac{5}{42}$
&
$(5,2),(3,4)$ or $(4,3),(2,5)$
\\
$F_{14}$
&
$\frac{1}{21}$
&
no solutions
\\
\bottomrule
\end{tabular}
\end{center}

Thus $F_{14}$ cannot be extracted. It remains to exclude $F_7$.

\medskip
\noindent\emph{Step 4. Use Condition (3) to rule out the extraction of $F_7$.}

\begin{claim}\label{claim:no-F7-extraction}
    The curve $F_7$ is not extracted by $h$.
\end{claim}

\begin{proof}
Suppose that $F_7$ is extracted. By the preceding table, the
primitive generators of the two new rays in the local model at $P$, ordered by increasing slope, are
\begin{align*}
    (5,2),(3,4)
    \qquad\text{or}\qquad
    (4,3),(2,5).
\end{align*}
Let $F_1$ and $F_2$ be the corresponding exceptional curves. Then
\begin{align*}
    R_h=\frac{1}{3}F_7+F_1+F_2.
\end{align*}

Using the weighted coordinates as in Lemma \ref{lem:rank-one-normal-forms}, \ref{lem:A26-local-coordinates} and \ref{lem:I22-local-coordinates}. And replacing $x_2$ by $-x_2$ in the $(I_2^2)$ case if necessary, we may assume that $P=[1:1:0]$.

Let
\begin{align*}
    \Gamma:=(u=0)=(x_2-x_1^2=0)\in|\mathcal O_S(2)|,
    \qquad
    \Gamma_Y:=h_*^{-1}\Gamma.
\end{align*}
The curve $\Gamma$ is an irreducible curve and passes through $Q_3$ and $P$.

We shall prove that 
\begin{align}\label{eq:Gamma-is-L-trivial}
    L\cdot\Gamma_Y=0.
\end{align}
Indeed, by the definition of $R_h$ and the projection formula,
\begin{align*}
    L\cdot \Gamma_Y&=h^*(K_S+C_1+C_2)\cdot \Gamma_Y-R_h\cdot \Gamma_Y\\
    &=(K_S+C_1+C_2)\cdot \Gamma-(\frac{1}{3}F_7+F_1+F_2)\cdot \Gamma_Y\\
    &=\frac{1}{3}-(\frac{1}{3}F_7+F_1+F_2)\cdot \Gamma_Y,
\end{align*}
It therefore suffices to prove that
\begin{align}\label{eq:Gamma-Rh-intersection}
    (\frac{1}{3}F_7+F_1+F_2)\cdot \Gamma_Y=\frac{1}{3}.
\end{align}

We first compute the contribution over $Q_3$. Since $\Gamma=(u=0)$, Corollary \ref{cor:F7-order-u} gives
\begin{align*}
    \ord_{F_7}(\Gamma)=\ord_{F_7}(u)=\frac{2}{3}.
\end{align*}
Since $F_7$ is the only exceptional curve lying over $Q_3$, the part of $h^*\Gamma-\Gamma_Y$ supported over $Q_3$ is $\frac{2}{3}F_7$. By Lemma~\ref{lem:toric-intersections}, $F_7^2=-\frac{3}{14}$. It follows that
\begin{align}\label{eq:Gamma-F7-intersection}
    \Gamma_Y\cdot F_7=\big(h^*\Gamma-\frac{2}{3}F_7 \big)\cdot F_7=-\frac{2}{3}F_7^2=\frac{1}{7}.
\end{align}

We next compute the contribution over $P$. By Lemma~\ref{lem:A26-local-coordinates} in the case $(A_2^6)$ and Lemma~\ref{lem:I22-local-coordinates} in the case $(I_2^2)$, there are analytic coordinates $(r,w)$ at $P$ such that the two analytic branches of $C_1+C_2$ are given by $r=0$ and $w=0$. Up to multiplication by a unit, the local equation $u$ of $\Gamma$ is
\begin{align*}
    w-r=0
    \quad\text{in the case $(A_2^6)$},
    \qquad
    r+w=0
    \quad\text{in the case $(I_2^2)$},
\end{align*}
For an exceptional curve $E$ corresponding to the ray generated by $(p,q)$ with respect to the analytic coordinates $(r,w)$,
\begin{align}\label{eq:Gamma-order-over-P}
    \ord_{E}(\Gamma)
    =
    \ord_{E}(r\pm w)
    =
    \min\{p,q\}.
\end{align}
Interchanging $r$ and $w$ exchanges the two analytic branches and preserves $r\pm w$ up to sign. It also exchanges the two configurations above and the labels $F_1,F_2$. We may therefore assume that the ordered generators are $(5,2),(3,4)$. By \eqref{eq:Gamma-order-over-P},
\begin{align*}
    \ord_{F_1}(\Gamma)=2,
    \qquad
    \ord_{F_2}(\Gamma)=3.
\end{align*}
Therefore, the part of $h^*\Gamma-\Gamma_Y$ supported over $P$ is $2F_1+3F_2$. By Lemma~\ref{lem:toric-intersections},
\begin{align*}
    F_1^2=-\frac{1}{7},
    \qquad
    F_1\cdot F_2=\frac{1}{14},
    \qquad
    F_2^2=-\frac{5}{42}.
\end{align*}
Since $h^*\Gamma\cdot(F_1+F_2)=0$, we obtain
\begin{align*}
    \Gamma_Y\cdot(F_1+F_2)
    &=
    -(2F_1+3F_2)\cdot(F_1+F_2)\\
    &=
    -\left(
        2\left(-\frac17\right)
        +5\left(\frac1{14}\right)
        +3\left(-\frac5{42}\right)
    \right)\\
    &=
    \frac27.
\end{align*}
The same equality holds for the other configuration by interchanging $r$ and $w$.

Combining the contributions over $Q_3$ and $P$, we obtain
\begin{align*}
    R_h\cdot \Gamma_Y&=(\frac{1}{3}F_7+F_1+F_2)\cdot \Gamma_Y\\
    &=\frac{1}{3}\cdot \frac{1}{7}+\frac{2}{7}=\frac{1}{3}.
\end{align*}
Consequently,
\begin{align*}
    L\cdot\Gamma_Y
    =\frac{1}{3}-R_h\cdot\Gamma_Y
    =0.
\end{align*}
Since $\Gamma$ is distinct from $C_1$ and $C_2$, its strict transform $\Gamma_Y$ is distinct from both $B_Y$ and $E$. This contradicts Condition~{\rm(3)}. Therefore $F_7$ is not extracted by $h$.
\end{proof}

Together with the preceding exclusions, we have proved that $h$ is an isomorphism over $S\setminus\{P\}$ and that the two new rays over $P$ have primitive generators $(5,2)$ and $(2,5)$, ordered by increasing slope.

In the $(A_2^6)$ case, the strict transforms of $C_{(1)}$ and $C_{(6)}$ still meet, so $h$ is an isomorphism over their intersection. Hence $P$ is the node of $C_{(6)}$. In the $(I_2^2)$ case, $P$ is one of the two crossings of $C_{(3)}$ and $C_{(4)}$. This proves the lemma.
\end{proof}

We now apply Lemma \ref{lem:exhaustion} to equality pairs.

\begin{proposition}\label{prop:reduction-to-two-types}
Let $(X,B)$ be an equality pair, and choose $h\colon Y\to S$ as in Lemma~\ref{lem:rank-one-reduction}. Then $h$ satisfies the assumptions of Lemma~\ref{lem:exhaustion}. In particular, it has one of the two possible extraction configurations listed there.
\end{proposition}

\begin{proof}
    Let $(X, B)$ be an equality pair. Let $f\colon Y\to X$ be the morphism constructed in Lemma \ref{lem:boundary-center-extraction}. By Lemma~\ref{lem:rank-one-reduction}, there is a birational morphism $h\colon Y\to S$ such that
\begin{align}\label{eq:six-sevenths-crepant}
    K_Y+\frac{6}{7}(B_Y+E)=
    h^*\left(K_S+\frac{6}{7}(C_1+C_2)\right),
\end{align}
where $\left(S,\frac{6}{7}(C_1+C_2)\right)$ is of type $(A_2^6)$ or $(I_2^2)$. And Lemma \ref{lem:h-toroidal} show that $h$ is toroidal with respect to the chosen completion $D$. This verifies Condition~{\rm(1)} of Lemma~\ref{lem:exhaustion}.

Since $K_Y+B_Y+E=f^*(K_X+B)$ and $K_X+B$ is ample, the divisor $L:=K_Y+B_Y+E$ is big and nef. Moreover $L^2=\frac{1}{42}$. This verifies Condition~{\rm(2)}.

Since $E$ is the only $f$-exceptional curve, we have $L\cdot E=0$. On the other hand, if $C\subset Y$ is an irreducible curve distinct from $E$, then $f_*C$ is a curve. By the ampleness of $K_X+B$, 
\begin{align*}
    L\cdot C=(K_X+B)\cdot f_*C>0.
\end{align*}
Thus $E$ is the unique irreducible curve on $Y$ having intersection zero with $L$. Furthermore,
\begin{align*}
    L\cdot B_Y=(K_X+B)\cdot B=\frac{1}{6}.
\end{align*}
This verifies Condition~{\rm(3)}.

We next verify the additional hypothesis concerning $C_{(6)}$ in the case $(A_2^6)$. The strict transform of $C_{(6)}$ on $Y$ is one of the rational curves $B_Y$ and $E$. Consequently, $C_{(6)}$ is integral and its normalization is $\bP^1$. On the other hand, $C_{(6)}$ is Cartier. Thus adjunction formula gives
\begin{align*}
    2p_a(C_{(6)})-2=(K_S+C_{(6)})\cdot C_{(6)}=0,
\end{align*}
and therefore $C_{(6)}$ has arithmetic genus one. Thus $C_{(6)}$ has a unique singular point. Since the curve $C_{(6)}$ in the case $(A_2^6)$ avoids the singular points of $S$, its unique singularity is either a node or an ordinary cusp at a smooth point of $S$. In the latter case, $\lct(S, 0; C_{(6)})=\frac{5}{6}$, contradicting the log canonicity of $(S, \frac{6}{7}C_{(6)})$. Hence $C_{(6)}$ is a nodal rational curve.

It remains to verify Condition~{\rm(4)}. In the $(A_2^6)$ case, $C_{(1)}$ and $C_{(6)}$ meet at one smooth point. Since their strict transforms $B_Y$ and $E$ still meet, $h$ is an isomorphism over this point. On the other hand, the strict transform of $C_{(6)}$ on $Y$ is smooth, whereas $C_{(6)}$ is nodal. Hence $h$ is not an isomorphism over the node of $C_{(6)}$. Thus the node is the unique smooth two-branch point over which $h$ is not an isomorphism.

In the $(I_2^2)$ case, $C_{(3)}$ and $C_{(4)}$ meet at two distinct smooth points. Since $B_Y$ and $E$ meet at exactly one point, $h$ is an isomorphism over one crossing and is not an isomorphism over the other. Therefore Condition~{\rm(4)} of Lemma~\ref{lem:exhaustion} holds.

We have now verified all the assumptions of Lemma~\ref{lem:exhaustion}. Hence $h$ has one of the two possible extraction configurations listed there.
\end{proof}

\subsection{A criterion for the ample model}

We establish a criterion for identifying the ample model of $K_Y+B_Y+E$ in the two cases of Lemma~\ref{lem:exhaustion}.

We retain the notation and conclusions of the preceding steps. Suppose that $h\colon (Y,B_Y+E)\to(S,C_1+C_2)$ is one of the two extractions listed in Lemma~\ref{lem:exhaustion}. We have
\begin{align}\label{eq:L-square}
    (K_Y+B_Y+E)^2=\frac{1}{42},\quad
    (K_Y+B_Y+E)\cdot B_Y=\frac{1}{6},\quad
    (K_Y+B_Y+E)\cdot E=0,
\end{align}
and $(K_Y+B_Y+E)\cdot F>0$ for every prime $h$-exceptional curve $F$. By assumption, the divisor $K_S+C_1+C_2\sim_\bQ \cO_S(1)$ is ample. For every prime $h$-exceptional curve $G$, set
\begin{align*}
    \lambda_F:=1-a(F,S,C_1+C_2),\qquad \Lambda:=\sum_F\lambda_F F.
\end{align*}
Then $K_Y+B_Y+E=h^*(K_S+C_1+C_2)-\Lambda$. 

Our aim is to identify the ample model of $K_Y+B_Y+E$ with the
hypersurface $X_{18}\subseteq\bP(2,6,7,9)$. In the next subsection, we will construct sections
\begin{align*}
    s_d\in H^0\bigl(S,d(K_S+C_1+C_2)\bigr)
    \simeq H^0(S,\cO_S(d)),
    \qquad d\in\{2,6,7,9\},
\end{align*}
and lift them to sections
$\tilde s_d\in H^0\bigl(Y,d(K_Y+B_Y+E)\bigr)$, which define a rational map from $Y$ to $\bP(2,6,7,9)$. We first establish the lemmas and the ample model criterion in Proposition~\ref{prop:ample-model-criterion}.

\begin{lemma}\label{lem:section-lifting}
Suppose that the zero divisors $(s_d=0)$, for $d\in\{2,6,7,9\}$, have no common irreducible component and that
\begin{align}\label{eq:finite-order-check}
    \min_{d\in\{2,6,7,9\}}
    \frac{\ord_F(s_d)}{d}
    =
    \lambda_F
\end{align}
for every prime $h$-exceptional curve $F$. Then each $s_d$ induces a section
\begin{align*}
    \tilde s_d\in H^0\bigl(Y,d(K_Y+B_Y+E)\bigr),
\end{align*}
and the zero divisors $(\tilde s_d=0)$ have no common irreducible component.
\end{lemma}

\begin{proof}
For every $d\in\{2,6,7,9\}$ and every prime $h$-exceptional curve
$G$, Equation~\eqref{eq:finite-order-check} gives
\begin{align*}
    \ord_G(s_d)\geq d\lambda_G.
\end{align*}
Since $K_Y+B_Y+E=h^*(K_S+C_1+C_2)-\sum_G\lambda_GG$, viewing $h^*s_d$ as a rational section of $\mathcal O_Y(d(K_Y+B_Y+E))$, the preceding inequalities show that it is regular. We denote the resulting section by
\begin{align*}
    \tilde s_d\in H^0\bigl(Y, d(K_Y+B_Y+E)\bigr).
\end{align*}
Its zero divisor is
\begin{align*}
    (\tilde s_d=0)
    =
    h_*^{-1}(s_d=0)
    +
    \sum_G\bigl(\ord_G(s_d)-d\lambda_G\bigr)G.
\end{align*}
In particular, $\ord_G(\tilde s_d)=\ord_G(s_d)-d\lambda_G$. Therefore
\begin{align*}
    \min_{d\in\{2,6,7,9\}}
    \frac{\ord_G(\tilde s_d)}{d}
    =0.
\end{align*}
Thus no $h$-exceptional curve is a common component of the zero divisors $(\tilde s_d=0)$.

If an irreducible curve $C$ on $Y$ not contracted by $h$ were a common component of the zero divisors $(\tilde s_d=0)$, then $h_*C$ would be a common component of all the zero divisors $(s_d=0)$, contrary to the assumption.
\end{proof}

\begin{lemma}\label{lem:equal-square}
Let $\varphi\colon Y\dashrightarrow X$ be a birational map between
normal projective surfaces, let $L$ be a $\bQ$-Cartier divisor on
$Y$, and let $A$ be an ample $\bQ$-Cartier divisor on $X$.
Take a common resolution of $\varphi$:
\begin{center}
\begin{tikzpicture}[baseline=(current bounding box.center)]
    \node (W) at (0,1.2) {$W$};
    \node (Y) at (-1.6,0) {$Y$};
    \node (X) at (1.6,0) {$X$};
    \draw[->] (W)--node[left] {$p$} (Y);
    \draw[->] (W)--node[right] {$q$} (X);
    \draw[->,dashed] (Y)--node[below] {$\varphi$} (X);
\end{tikzpicture}
\end{center}
Suppose that there is an effective $p$-exceptional $\bQ$-divisor $F$
such that
\begin{align}\label{eq:equal-square-pullback}
    p^*L\sim_\bQ q^*A+F.
\end{align}
If $L^2=A^2$, then $F=0$, $\varphi$ is a birational morphism and $L\sim_\bQ\varphi^*A$. 
\end{lemma}

\begin{proof}
Since $F$ is $p$-exceptional, $p^*L\cdot F=0$. Intersecting \eqref{eq:equal-square-pullback} with $F$, we obtain $q^*A\cdot F+F^2=0$. Hence
\begin{align*}
    L^2=(p^*L)^2=p^*L\cdot(q^*A+F)=(q^*A+F)\cdot q^*A=A^2+q^*A\cdot F.
\end{align*}
Since $A$ is ample and $F\geq0$, we have $q^*A\cdot F\geq0$. The equality $L^2=A^2$ gives $q^*A\cdot F=0$, and therefore $F^2=0$. The intersection matrix of the $p$-exceptional curves is negative definite. Since $F$ is $p$-exceptional, it follows that $F=0$. Thus $p^*L\sim_\bQ q^*A$.

For every $p$-exceptional curve $C$, we have
\begin{align*}
    q^*A\cdot C=p^*L\cdot C=0.
\end{align*}
Since $A$ is ample, $q$ contracts every fiber of $p$ and hence factors through $p$ by \cite[Lemma~1.15]{Debarre01}. Thus $\varphi\colon Y\to X$ is a birational morphism, and
\begin{align*}
    p^*L\sim_\bQ q^*A=p^*\varphi^*A.
\end{align*}
It follows that $L\sim_\bQ\varphi^*A$.
\end{proof}

\begin{proposition}\label{prop:ample-model-criterion}
Retain the notation above. Suppose that sections
\begin{align*}
    s_d\in H^0(S, d(K_S+C_1+C_2)),
    \qquad
    d\in\{2,6,7,9\},
\end{align*}
satisfy the following conditions:
\begin{enumerate}[label={\rm(\arabic*)}]
    \item Their zero divisors have no common irreducible component.
    \item For every prime $h$-exceptional curve $F$,
    \begin{align*}
        \min_{d\in\{2,6,7,9\}}\frac{\ord_F(s_d)}{d}=\lambda_F.
    \end{align*}
    \item the sections satisfy the defining equation of $X_{18}$, and the rational map
    \begin{align*}
        \Phi_S:=[s_2:s_6:s_7:s_9]
        \colon S\dashrightarrow X_{18}
    \end{align*}
    is birational.
\end{enumerate}
Set $A:=\cO_{\bP(2,6,7,9)}(1)|_{X_{18}}$. Then the lifted sections induce a birational morphism $\varphi\colon Y\to X_{18}$ such that
\begin{align*}
    K_Y+B_Y+E\sim_\bQ\varphi^*A.
\end{align*}
In particular, $X_{18}$ is the ample model of $K_Y+B_Y+E$.
\end{proposition}

\begin{proof}
Set $L:=K_Y+B_Y+E$. By Conditions~{\rm(1)} and~{\rm(2)} and
Lemma~\ref{lem:section-lifting}, the sections $s_d$ induce sections
$\tilde s_d\in H^0(Y,\cO_Y(dL))$, for $d\in\{2,6,7,9\}$, whose zero divisors have no common irreducible component.

These sections define a rational map
\begin{align*}
    \varphi:=[\tilde s_2:\tilde s_6:\tilde s_7:\tilde s_9]
    \colon Y\dashrightarrow\bP(2,6,7,9).
\end{align*}
By construction, $\varphi=\Phi_S\circ h$ as rational maps, so its image is contained in $X_{18}$. Since $h$ and $\Phi_S$ are birational, so is $\varphi\colon Y\dashrightarrow X_{18}$.

Choose a sufficiently large and divisible positive integer $m$ such that $mL$ is Cartier, $126=\operatorname{lcm}(2,6,7,9)\mid m$, and $\cO_{\bP(2,6,7,9)}(m)$ is very ample. Let $\iota_m\colon X_{18}\hookrightarrow\bP^N$ be the restriction to $X_{18}$ of the resulting embedding. Then $\iota_m^*\cO_{\bP^N}(1)=mA$.

Consider all monomial sections in $\tilde s_2,\tilde s_6,\tilde s_7,\tilde s_9$ of degree $m$:
\begin{align*}
    \tilde s_2^{e_2}
    \tilde s_6^{e_6}
    \tilde s_7^{e_7}
    \tilde s_9^{e_9}
    \in H^0(Y,mL),
    \qquad
    2e_2+6e_6+7e_7+9e_9=m.
\end{align*}
By the definition of $\varphi$, the rational map given by these sections is $\iota_m\circ\varphi$. These sections include $\tilde s_2^{m/2},\tilde s_6^{m/6},\tilde s_7^{m/7}, \tilde s_9^{m/9}$. Since the zero divisors
\begin{align*}
    (\tilde s_2=0),\quad
    (\tilde s_6=0),\quad
    (\tilde s_7=0),\quad
    (\tilde s_9=0)
\end{align*}
have no common irreducible component, the same holds for the zero divisors of these monomial sections.

Take a common resolution of $\varphi$ such that $p\colon W\to Y$ also resolves the common zeros of the monomial sections above:
\begin{center}
\begin{tikzpicture}[baseline=(current bounding box.center)]
    \node (W) at (0,1.2) {$W$};
    \node (Y) at (-1.6,0) {$Y$};
    \node (X) at (1.6,0) {$X_{18}$};
    \draw[->] (W)--node[left] {$p$} (Y);
    \draw[->] (W)--node[right] {$q$} (X);
    \draw[->,dashed] (Y)--node[below] {$\varphi$} (X);
\end{tikzpicture}
\end{center}
Let $F_m$ be the greatest common divisor of the zero divisors of these sections after pullback to $W$. Since their zero divisors on $Y$ have no common irreducible component, $F_m$ is effective and $p$-exceptional. By the choice of $p$, after removing $F_m$, the resulting sections have no common zeros and define the morphism $\iota_m\circ q$. Therefore
\begin{align*}
    p^*(mL)\sim(\iota_m\circ q)^*\cO_{\bP^N}(1)+F_m=q^*(mA)+F_m.
\end{align*}
In particular,
\begin{align*}
    p^*L\sim_\bQ q^*A+\frac{1}{m}F_m.
\end{align*}

Since $L^2=A^2=\frac{1}{42}$, Lemma~\ref{lem:equal-square} shows that $\varphi$ is a birational morphism and $L\sim_\bQ\varphi^*A$. As $A$ is ample, $X_{18}$ is the ample model of $L=K_Y+B_Y+E$.
\end{proof}

\subsection{The ample model and its boundary}

\begin{proposition}\label{prop:ample-model-two-types}
For each of the two extraction types in Lemma~\ref{lem:exhaustion}, the ample model of $K_Y+B_Y+E$ is $X_{18}$.
\end{proposition}

\begin{proof}
We now treat the two extraction types separately. We use the notation and the fixed local model at the smooth $P$ of Lemma~\ref{lem:exhaustion}. In this fixed local model, the torus-invariant curves correspond to $(r=0),(w=0)$. Let $(p_1,q_1)$ and $(p_2,q_2)$ be the primitive generators of the two new rays in the local model at the smooth two-branch point $P$, ordered by increasing slope. Let $F_1$ and $F_2$ be the corresponding exceptional curves. We write $F_i=(p_i,q_i)$ for short.

\smallskip
\noindent
\emph{The case $(A_2^6)$: The extracted curves are $F_1=(5,2)$ and $F_2=(2,5)$.}

Write $S=\bP(1,2,3)_{[x_1:x_2:x_3]}$, where $\deg x_i=i$. By Lemma~\ref{lem:rank-one-normal-forms}, we may choose weighted coordinates on $S$ such that if we set
\begin{align*}
    u:=x_2-x_1^2, \qquad G_6:=x_3^2-3x_1^2u^2-u^3,
\end{align*}
then $C_{(1)}=(x_1=0)$ and $C_{(6)}=(G_6=0)$. The curve $C_{(6)}$ has a node at $P=[1:1:0]$.

Set
\begin{align}
    s_2&=-u,&
    s_6&=G_6,&
    s_7&=\sqrt3\,x_1G_6,&
    s_9&=\mathrm{i}x_3G_6.
\end{align}
Their common zero locus is the single point $P$. A direct substitution gives
\begin{align}
    s_9^2+s_6^2(s_6-s_2^3)+s_2^2s_7^2=0.
\end{align}
Hence these sections define a rational map
\begin{align*}
    \Phi_S:=[s_2:s_6:s_7:s_9]\colon S\dashrightarrow X_{18}.
\end{align*}
On the dense open set $(s_6\neq0)$, the original weighted coordinates are recovered from
\begin{align*}
    x_1=\frac{s_7}{\sqrt{3}\,s_6}, \quad
    x_2=-s_2+\frac{s_7^2}{3s_6^2}, \quad
    x_3=-\frac{\mathrm{i}s_9}{s_6}.
\end{align*}
Thus $\Phi_S$ is birational onto its image. Its image has dimension two and is therefore dense in the irreducible surface $X_{18}$. Consequently, $\Phi_S$ is birational.

It remains to check the orders along the two extracted curves. Using the ordered analytic coordinates $(r,w)$ introduced in Lemma \ref{lem:A26-local-coordinates} at $P$, the sections have the following expressions:
\begin{align*}
    s_2&=\frac{r-w}{2\theta},&
    s_6&=rw,&
    s_7&=\sqrt3\,rw,&
    s_9&=\frac{\mathrm{i}}2(r^2w+rw^2).
\end{align*}
Lemma~\ref{lem:toric-orders} gives
\begin{center}
\renewcommand{\arraystretch}{1.2}
\begin{tabular}{@{}ccccc@{}}
\toprule
$F$
& $\ord(s_2)/2$
& $\ord(s_6)/6$
& $\ord(s_7)/7$
& $\ord(s_9)/9$\\
\midrule
$F_1=(5,2)$ & $1$ & $7/6$ & $1$ & $1$\\
$F_2=(2,5)$ & $1$ & $7/6$ & $1$ & $1$\\
\bottomrule
\end{tabular}
\end{center}
For both extracted curves, $\lambda_F=1$. Hence
\begin{align*}
    \min_{d\in\{2,6,7,9\}}
    \frac{\ord_F(s_d)}{d}
    =
    \lambda_F.
\end{align*}
Thus all the assumptions of Proposition~\ref{prop:ample-model-criterion} are satisfied in this case.

\smallskip
\noindent
\emph{The case $(I_2^2)$: The extracted curves are $F_1=(5,2)$ and $F_2=(2,5)$.}

Write $S=\bP(1,2,3)_{[x_1:x_2:x_3]}$, where $\deg x_i=i$. By Lemma~\ref{lem:rank-one-normal-forms}, we may choose weighted coordinates on $S$ such that if we set
\begin{align*}
    u:=x_2-x_1^2, \qquad G_4:=x_2^2-x_1^4-x_1x_3=u^2+2x_1^2u-x_1x_3,
\end{align*}
then $C_{(3)}=(x_3=0)$ and $C_{(4)}=(G_4=0)$. They meet transversely at $[1:1:0]$ and $[1:-1:0]$. Since the involution
$x_2\mapsto-x_2$ exchanges these points and preserves the
boundary, we may work at $P=[1:1:0]$.

Choose constants $c,d\in\bC^*$ such that
\begin{align*}
    c^3=-2,\qquad c^2 d^2=-1.
\end{align*} 
and set
\begin{align}\label{eq:I-sections-one}
    s_2:=cu, \quad s_6:=-x_3(x_3-2x_1u), \quad s_7:=-dx_3G_4, \quad s_9:=x_3(x_3^2-4x_2u^2)+3ux_3G_4.
\end{align}
Their common zero locus is the single point $P$. A direct substitution gives
\begin{align}
    s_9^2+s_6^2(s_6-s_2^3)+s_2^2s_7^2=0.
\end{align}
Hence these sections define a rational map
\begin{align*}
    \Phi_S:=[s_2:s_6:s_7:s_9]\colon S\dashrightarrow X_{18}.
\end{align*}

We claim that $\Phi_S$ is birational. Set
\begin{align*}
    t_2:=u, \quad t_6:=x_3(x_3-2x_1u), \quad
    t_7:=-x_3G_4, \quad t_9:=x_3(x_3^2-4x_2u^2).
\end{align*}
Then
\begin{align*}
    s_2=ct_2,\qquad
    s_6=-t_6,\qquad
    s_7=dt_7,\qquad
    s_9=t_9-3t_2t_7.
\end{align*}
Hence $\Phi_S=\tau\circ[t_2:t_6:t_7:t_9]$, where
\begin{align*}
    \tau[T_2:T_6:T_7:T_9]:=[cT_2:-T_6:dT_7:T_9-3T_2T_7]
\end{align*}
is an automorphism of $\bP(2,6,7,9)$. On a dense open set $(t_2t_6(t_6-2t_2^3)\neq 0)$, the map $[t_2:t_6:t_7:t_9]$ admits the rational inverse
\begin{align*}
    x_1=
    \frac{t_2^2t_9+t_7(t_6-4t_2^3)}
         {t_6(t_6-2t_2^3)}, \qquad
    x_2=t_2+x_1^2, \qquad 
    x_3=\frac{x_1t_6-t_7}{t_2^2}.
\end{align*}
Therefore $\Phi_S$ is birational.

It remains to check the orders along the two extracted curves. Using the ordered analytic coordinates $(r,w)$ introduced in Lemma \ref{lem:I22-local-coordinates} at $P$, set $A:=2+r+w$, which is a unit at $P$. The sections have the following expressions:
\begin{align*}
    s_2&=c(r+w),\quad 
    &s_6&=-Ar(r^2+rw-2w),\\
    s_7&=-dA^2rw,\quad
    &s_9&=Ar\left(
        (r+w)^4
        +Aw\bigl(w-r-(r+w)(2r+w)\bigr)
    \right).
\end{align*}
Lemma~\ref{lem:toric-orders} gives
\begin{center}
\renewcommand{\arraystretch}{1.2}
\begin{tabular}{@{}ccccc@{}}
\toprule
$F$
& $\ord(s_2)/2$
& $\ord(s_6)/6$
& $\ord(s_7)/7$
& $\ord(s_9)/9$\\
\midrule
$F_1=(5,2)$ & $1$ & $7/6$ & $1$ & $1$\\
$F_2=(2,5)$ & $1$ & $1$   & $1$ & $1$\\
\bottomrule
\end{tabular}
\end{center}
For both extracted curves, $\lambda_F=1$. Hence
\begin{align*}
    \min_{k\in\{2,6,7,9\}}\frac{\ord_F(s_k)}{k}=\lambda_F.
\end{align*}
This verifies all the assumptions of Proposition~\ref{prop:ample-model-criterion} in this case. 

In either case, Proposition~\ref{prop:ample-model-criterion} shows that the lifted sections induce a birational morphism
\begin{align*}
    \varphi\colon Y\to X_{18}, \qquad L\sim_{\bQ}\varphi^*\mathcal O_{X_{18}}(1).
\end{align*}
Hence $X_{18}$ is the ample model of $L$, and $\varphi$ is the corresponding morphism. 
\end{proof}

It remains to identify the boundary on this model.

\begin{proposition}\label{prop:ample-model-boundary}
For each of the two cases above,
\begin{align*}
    \varphi_*B_Y=B_{18},
    \qquad
    \varphi_*E=0.
\end{align*}
\end{proposition}

\begin{proof}
Recall that $B_{18}={X_{18}}\cap (t=0)$. Under $\Phi_S=[s_2:s_6:s_7:s_9]$, the coordinate $t$ pulls back to a nonzero scalar multiple of $s_7$. The zero divisor $(s_7=0)$ is
\begin{align*}
    C_{(1)}+C_{(6)}=(x_1G_6=0) &\qquad \text{ in the } (A_2^6) \text{ case}\\
    C_{(3)}+C_{(4)}=(x_3G_4=0) &\qquad\text{ in the } (I_2^2) \text{ case}
\end{align*}
Consequently, $\tilde s_7$ vanishes identically along both $B_Y$ and $E$, so $\varphi(B_Y)$ and $\varphi(E)$ are contained in $B_{18}$. By Proposition~\ref{prop:ample-model-two-types}, the projection formula, and Equation~\eqref{eq:L-square},
\begin{align*}
    \mathcal O_{X_{18}}(1)\cdot\varphi_*E
    &=L\cdot E=0,\\
    \mathcal O_{X_{18}}(1)\cdot\varphi_*B_Y
    &=L\cdot B_Y=\frac{1}{6}.
\end{align*}
Since $\mathcal O_{X_{18}}(1)$ is ample, it follows that $\varphi_*E=0,\varphi_*B_Y\neq0$. Thus $\varphi(B_Y)$ is a curve contained in $B_{18}$. By Proposition \ref{prop:X18-structure}, $B_{18}$ is a smooth rational curve. Therefore $\varphi(B_Y)=B_{18}$. Since $\varphi$ is a birational morphism between normal surfaces, $\varphi_*B_Y=B_{18}$.
\end{proof}

\begin{remark}
Both extraction configurations in Lemma~\ref{lem:exhaustion} are realized. The $(I_2^2)$ configuration recovers the construction in Proposition~\ref{prop:construction}, while the $(A_2^6)$ configuration gives the same ample model $X_{18}$ and the same boundary $B_{18}$ by the preceding propositions.
\end{remark}

\subsection{Proof of Theorem~\ref{thm:uniqueness}}

\begin{proof}[Proof of Theorem~\ref{thm:uniqueness}]
Let $(X,B)$ be an equality pair. By Proposition~\ref{prop:reduction-to-two-types}, its associated morphism $h\colon Y\to S$ has one of the two extraction types in Lemma~\ref{lem:exhaustion}. Proposition~\ref{prop:ample-model-two-types} shows that the ample model of $K_Y+B_Y+E$ is $X_{18}$.

On the other hand, $K_Y+B_Y+E=f^*(K_X+B)$. So $X$ is the ample model of the same divisor. Hence $X\simeq X_{18}$. Finally, Proposition~\ref{prop:ample-model-boundary} identifies $B=f_*(B_Y+E)$ with $B_{18}$. Therefore $(X,B)\simeq(X_{18},B_{18})$.
\end{proof}

\begin{remark}
Let
\begin{align*}
    C_A:=C_{(1)}+C_{(6)},\qquad C_I:=C_{(3)}+C_{(4)}
\end{align*}
be the boundaries in the cases $(A_2^6)$ and $(I_2^2)$, respectively.

Although the two rank-one exceptional pairs are not isomorphic, our construction above shows that they belong to the same crepant birational log Calabi--Yau class. Indeed, for $\star\in\{A,I\}$, the corresponding extraction satisfies
\begin{align*}
    K_{Y_\star}+\frac{6}{7}(B_{Y_\star}+E_\star)
    &={h_\star}^*
    \left(K_{S_\star}+\frac{6}{7}C_\star\right)\\
    &={f_\star}^*
    \left(K_{X_{18}}+\frac{6}{7}B_{18}\right).
\end{align*}

Moreover, the two models $(Y_A,B_{Y_A}+E_A)$ and $(Y_I,B_{Y_I}+E_I)$ are isomorphic. Indeed, the preceding equalities give
\begin{align*}
    a\left(E_A,X_{18},\frac67B_{18}\right)
    =
    a\left(E_I,X_{18},\frac67B_{18}\right)
    =
    \frac17.
\end{align*}
By Theorem~\ref{thm:equality-properties}, $\delta\left(X_{18},\frac67B_{18}\right)=2$, while $B_{18}$ itself has log discrepancy $\frac17$. Hence $E_A$ and $E_I$ define the same valuation. Since both $f_A$ and $f_I$ extract only this valuation, the induced birational map $Y_A\dashrightarrow Y_I$ is an isomorphism in codimension one. As there are no nontrivial small birational maps between normal surfaces, it is an isomorphism over $X_{18}$, preserving the corresponding boundaries.

Identifying these two models with $(Y,B_Y+E)$, we obtain the following diagram, in which all arrows are crepant:
\begin{center}
\begin{tikzcd}[column sep=2.5em,row sep=3.5em]
& \left(Y,\frac67(B_Y+E)\right)
  \arrow[dl,"h_A"']
  \arrow[d,"f"]
  \arrow[dr,"h_I"] & \\
\left(S_A,\frac67C_A\right)
&
\left(X_{18},\frac67B_{18}\right)
&
\left(S_I,\frac67C_I\right).
\end{tikzcd}
\end{center}
Thus $(A_2^6)$ and $(I_2^2)$ not only belong to the same crepant birational log Calabi--Yau class, but also give two distinct Mori fiber space of $\big(K_Y+\frac{6}{7}B+(\frac{6}{7}-\epsilon)E\big)$-MMP started from $Y$.
\end{remark}

\section{Proof of Theorem \ref{thm:main-thm}}

\begin{proof}[Proof of Theorem \ref{thm:main-thm}]
By the lower bound, Proposition~\ref{prop:construction}, and Theorem~\ref{thm:uniqueness}, it remains only to show that $\frac{1}{42}$ is an infinite accumulation point of $\operatorname{Vol}(2,\{0\})$.

Since $X$ is klt, it is $\bQ$-factorial. The equality pair $(X,B)$ satisfies the assumptions of \cite[Theorem~1.4]{Shao26p}. Hence its volume $\frac{1}{42}$ is an infinite accumulation point of $\operatorname{Vol}(2,\{0,1\})$. 

We recall the construction in the proof of \cite[Theorem~1.4]{Shao26p}. Let $f\colon Y\to X$, $B_Y$, and $E$ be as in Lemma~\ref{lem:boundary-center-extraction}. Apply the construction in \cite[Proposition~3.1]{Shao26p} to the pair $(Y,B_Y+E)$. By construction, $E$ is the only irreducible curve having intersection zero with $K_Y+B_Y+E$. Set $y:=B_Y\cap E$. Let $l$ be the integer used in the proof of \cite[Proposition~3.1]{Shao26p}. For each $n\geq1$, take $(p_1,\ldots,p_n)=(l+n,\ldots,l+1)$ and choose $q_1,\ldots,q_n$ successively as in that proof. The construction gives a morphism $\pi\colon W\to Y$ centered at $y$, together with a morphism $g\colon W\to Z$ contracting $C_0:=\pi_*^{-1}E$. Set
\begin{align*}
    C_{n+1}:=\pi_*^{-1}B_Y,
    \qquad
    B_Z:=g_*C_{n+1}.
\end{align*}
Since $g$ contracts only $C_0$, the curve $B_Z$ is a rational curve. By the construction in \cite[Proposition~3.1]{Shao26p}, $(Z,B_Z)$ is an lc surface pair, $K_Z+B_Z$ is big and nef. We define divisor $\overline B_W$ by $K_W+\overline B_W=g^*(K_Z+B_Z)$.

By the choice of $l$, we have $\frac16\geq\frac1l$, and hence $l\geq6$. Since $C_0\cdot C_{n+1}=0$, we obtain
\begin{align*}
    (K_Z+B_Z)\cdot B_Z
    &=(K_W+\overline B_W)\cdot C_{n+1}\\
    &=(K_Y+B_Y+E)\cdot B_Y-\frac1{p_n}\\
    &=\frac16-\frac1{p_n}>0.
\end{align*}
By abundance for lc surface pairs, $(Z,B_Z)$ has an ample model. The above computation shows that $B_Z$ is not contracted by the corresponding morphism. Replacing $(Z,B_Z)$ by its ample model and retaining the same notation, $(Z,B_Z)$ remains an lc surface pair, $B_Z$ remains a rational curve, and $K_Z+B_Z$ is ample. This replacement does not change the volume.

Since $B_Z$ is a rational curve and $(K_Z+B_Z)\cdot B_Z>0$, adjunction formula shows that $B_Z$ cannot be a smooth curve contained in the smooth locus of $Z$. Hence $B_Z$ is an accessible non-klt center. Moreover, $B_Z$ has coefficient one and geometric genus zero.

Let $\mathcal W$ be the set consisting of all volumes $\operatorname{vol}(Z,K_Z+B_Z)$ obtained from the above choices of parameters $(p_1,q_1,\cdots,p_n,q_n)$. Applying \cite[Theorem~1.1]{AlexeevLiu19accpoint} with coefficient set $\{0\}$ to these pairs, we obtain $\mathcal W\subseteq \mathrm{Vol}^{(1)}(2,\{0\})$. Repeatedly taking accumulation points gives $\mathcal W^{(n)} \subseteq \mathrm{Vol}^{(n+1)}(2,\{0\})$ for every $n\geq0$.

The iterated limit argument in the proof of \cite[Proposition~3.1]{Shao26p} shows that $\frac1{42}\in\mathcal W^{(n)}$ for every $n\geq1$. Therefore $\frac1{42}\in \mathrm{Vol}^{(\infty)}(2,\{0\})$ and hence $m_2\leq\frac1{42}$. This completes the proof.
\end{proof}

\bibliographystyle{amsalpha}
\bibliography{references}

@book {Kawakita24book,
    AUTHOR = {Kawakita, Masayuki},
     TITLE = {Complex algebraic threefolds},
    SERIES = {Cambridge Studies in Advanced Mathematics},
    VOLUME = {209},
 PUBLISHER = {Cambridge University Press, Cambridge},
      YEAR = {2024},
     PAGES = {xii+490},
      ISBN = {978-1-108-84423-9},
   MRCLASS = {14E30 (14-02 14E05 14E15 14J30 14J45)},
  MRNUMBER = {4653048},
MRREVIEWER = {Andreas\ H\"oring},
}

@article {Ambro06toricmld,
    AUTHOR = {Ambro, Florin},
     TITLE = {The set of toric minimal log discrepancies},
   JOURNAL = {Cent. Eur. J. Math.},
  FJOURNAL = {Central European Journal of Mathematics},
    VOLUME = {4},
      YEAR = {2006},
    NUMBER = {3},
     PAGES = {358--370},
      ISSN = {1895-1074,1644-3616},
   MRCLASS = {14M25 (14B05 14E15)},
  MRNUMBER = {2233855},
MRREVIEWER = {Harry\ Tamvakis},
       DOI = {10.2478/s11533-006-0013-x},
       URL = {https://doi.org/10.2478/s11533-006-0013-x},
}

@article {Shao26p,
    AUTHOR = {Shao, Weili},
     TITLE = {On infinite accumulation points of log canonical volumes},
   JOURNAL = {Bull. Lond. Math. Soc.},
  FJOURNAL = {Bulletin of the London Mathematical Society},
    VOLUME = {58},
      YEAR = {2026},
    NUMBER = {6},
     PAGES = {Paper No. e70401},
      ISSN = {0024-6093,1469-2120},
   MRCLASS = {14},
  MRNUMBER = {5079794},
       DOI = {10.1112/blms.70401},
       URL = {https://doi.org/10.1112/blms.70401},
}

@incollection {Shokurov00Complements-on-surfaces,
    AUTHOR = {Shokurov, V. V.},
     TITLE = {Complements on surfaces},
      NOTE = {Algebraic geometry, 10},
   JOURNAL = {J. Math. Sci. (New York)},
  FJOURNAL = {Journal of Mathematical Sciences (New York)},
    VOLUME = {102},
      YEAR = {2000},
    NUMBER = {2},
     PAGES = {3876--3932},
      ISSN = {1072-3374},
   MRCLASS = {14E30 (14J26 14J30)},
  MRNUMBER = {1794169},
MRREVIEWER = {I.\ Dolgachev},
       DOI = {10.1007/BF02984106},
       URL = {https://doi.org/10.1007/BF02984106},
}

@article {ChenHan21complements,
    AUTHOR = {Chen, Guodu and Han, Jingjun},
     TITLE = {Boundedness of {$(\epsilon, n)$}-complements for surfaces},
   JOURNAL = {Adv. Math.},
  FJOURNAL = {Advances in Mathematics},
    VOLUME = {383},
      YEAR = {2021},
     PAGES = {Paper No. 107703, 40},
      ISSN = {0001-8708,1090-2082},
   MRCLASS = {14E30 (14C20 14J17 14J27 14J45)},
  MRNUMBER = {4236289},
MRREVIEWER = {James\ McKernan},
       DOI = {10.1016/j.aim.2021.107703},
       URL = {https://doi.org/10.1016/j.aim.2021.107703},
}

@book {Kollar13book,
    AUTHOR = {Koll\'ar, J\'anos},
     TITLE = {Singularities of the minimal model program},
    SERIES = {Cambridge Tracts in Mathematics},
    VOLUME = {200},
      NOTE = {With a collaboration of S\'andor Kov\'acs},
 PUBLISHER = {Cambridge University Press, Cambridge},
      YEAR = {2013},
     PAGES = {x+370},
      ISBN = {978-1-107-03534-8},
   MRCLASS = {14E30 (14B05)},
  MRNUMBER = {3057950},
MRREVIEWER = {Tommaso\ De Fernex},
       DOI = {10.1017/CBO9781139547895},
       URL = {https://doi.org/10.1017/CBO9781139547895},
}

@article {BCHM,
    AUTHOR = {Birkar, Caucher and Cascini, Paolo and Hacon, Christopher D.
              and McKernan, James},
     TITLE = {Existence of minimal models for varieties of log general type},
   JOURNAL = {J. Amer. Math. Soc.},
  FJOURNAL = {Journal of the American Mathematical Society},
    VOLUME = {23},
      YEAR = {2010},
    NUMBER = {2},
     PAGES = {405--468},
      ISSN = {0894-0347,1088-6834},
   MRCLASS = {14E30 (14E05)},
  MRNUMBER = {2601039},
MRREVIEWER = {Mark\ Gross},
       DOI = {10.1090/S0894-0347-09-00649-3},
       URL = {https://doi.org/10.1090/S0894-0347-09-00649-3},
}

@book {Prokhorov01lectures,
    AUTHOR = {Prokhorov, Yuri G.},
     TITLE = {Lectures on complements on log surfaces},
    SERIES = {MSJ Memoirs},
    VOLUME = {10},
 PUBLISHER = {Mathematical Society of Japan, Tokyo},
      YEAR = {2001},
     PAGES = {viii+130},
      ISBN = {4-931469-12-4},
   MRCLASS = {14E30 (14B05 14E05 14J26 14J27)},
  MRNUMBER = {1830440},
MRREVIEWER = {Massimiliano\ Mella},
}

@article {Fujino12MMP,
    AUTHOR = {Fujino, Osamu},
     TITLE = {Minimal model theory for log surfaces},
   JOURNAL = {Publ. Res. Inst. Math. Sci.},
  FJOURNAL = {Publications of the Research Institute for Mathematical
              Sciences},
    VOLUME = {48},
      YEAR = {2012},
    NUMBER = {2},
     PAGES = {339--371},
      ISSN = {0034-5318,1663-4926},
   MRCLASS = {14E30},
  MRNUMBER = {2928144},
MRREVIEWER = {Paul\ A.\ Hacking},
       DOI = {10.2977/PRIMS/71},
       URL = {https://doi.org/10.2977/PRIMS/71},
}

@misc{LiuLiu26minimal,
      title={The minimal volume of surfaces of log general type with non-empty non-klt locus}, 
      author={Jihao Liu and Wenfei Liu},
      year={2023},
      eprint={2308.14268},
      archivePrefix={arXiv},
      primaryClass={math.AG},
      note = {\href{https://arxiv.org/abs/2308.14268}{arXiv:2308.14268},
        to appear in Algebraic Geometry}
}

@article {AlexeevLiu19accpoint,
    AUTHOR = {Alekseev, V. A. and Liu, V.},
     TITLE = {On accumulation points of volumes of log surfaces},
   JOURNAL = {Izv. Ross. Akad. Nauk Ser. Mat.},
  FJOURNAL = {Izvestiya Rossiiskoi Akademii Nauk. Seriya Matematicheskaya},
    VOLUME = {83},
      YEAR = {2019},
    NUMBER = {4},
     PAGES = {5--25},
      ISSN = {1607-0046,2587-5906},
   MRCLASS = {14J29 (14J26 14R05)},
  MRNUMBER = {3985688},
MRREVIEWER = {Kamil\ Rusek},
       DOI = {10.4213/im8842},
       URL = {https://doi.org/10.4213/im8842},
}

@book {Debarre01,
    AUTHOR = {Debarre, Olivier},
     TITLE = {Higher-dimensional algebraic geometry},
    SERIES = {Universitext},
 PUBLISHER = {Springer-Verlag, New York},
      YEAR = {2001},
     PAGES = {xiv+233},
      ISBN = {0-387-95227-6},
   MRCLASS = {14-02 (14E30 14Jxx)},
  MRNUMBER = {1841091},
MRREVIEWER = {Mark\ Gross},
       DOI = {10.1007/978-1-4757-5406-3},
       URL = {https://doi.org/10.1007/978-1-4757-5406-3},
}

@incollection {Nakayama17toric,
    AUTHOR = {Nakayama, Noboru},
     TITLE = {A variant of {S}hokurov's criterion of toric surface},
 BOOKTITLE = {Algebraic varieties and automorphism groups},
    SERIES = {Adv. Stud. Pure Math.},
    VOLUME = {75},
     PAGES = {287--392},
 PUBLISHER = {Math. Soc. Japan, Tokyo},
      YEAR = {2017},
      ISBN = {978-4-86497-048-8},
   MRCLASS = {14J25 (14J26 14M25)},
  MRNUMBER = {3793369},
MRREVIEWER = {Marcel\ Morales},
       DOI = {10.2969/aspm/07510287},
       URL = {https://doi.org/10.2969/aspm/07510287},
}

@misc{liuliu26minimalvolumestablesurfaces,
  author = {Jihao Liu and Wenfei Liu},
  title  = {The minimal volume of stable surfaces of rank one},
  year   = {2026},
  note = {\href{https://arxiv.org/abs/2605.05641}{arXiv:2605.05641}},
}

@book {Kollar92Utah,
    author    = {Koll{\'a}r, J{\'a}nos and others},
     TITLE = {Flips and abundance for algebraic threefolds},
      NOTE = {Papers from the Second Summer Seminar on Algebraic Geometry
              held at the University of Utah, Salt Lake City, Utah, August
              1991,
              Ast\'{e}risque No. 211 (1992)},
 PUBLISHER = {Soci\'{e}t\'{e} Math\'{e}matique de France, Paris},
      YEAR = {1992},
     PAGES = {1--258},
      ISSN = {0303-1179,2492-5926},
   MRCLASS = {14E30 (14E35 14M10)},
  MRNUMBER = {1225842},
MRREVIEWER = {Mark\ Gross},
}

@article{Kuwata99lct-red-curve,
  author     = {Kuwata, Takayasu},
  title      = {On log canonical thresholds of reducible plane curves},
  journal    = {Amer. J. Math.},
  fjournal   = {American Journal of Mathematics},
  volume     = {121},
  year       = {1999},
  number     = {4},
  pages      = {701--721},
  issn       = {0002-9327, 1080-6377},
  mrclass    = {14H20 (14B05 14E30)},
}

@article {Prokhorov02lct-II,
    AUTHOR = {Prokhorov, Yu. G.},
     TITLE = {On log canonical thresholds. {II}},
   JOURNAL = {Comm. Algebra},
  FJOURNAL = {Communications in Algebra},
    VOLUME = {30},
      YEAR = {2002},
    NUMBER = {12},
     PAGES = {5809--5823},
      ISSN = {0092-7872,1532-4125},
   MRCLASS = {14E30 (14J30)},
  MRNUMBER = {1941925},
MRREVIEWER = {Jaros\l aw\ A.\ Wi\'sniewski},
       DOI = {10.1081/AGB-120016015},
       URL = {https://doi.org/10.1081/AGB-120016015},
}

@misc{LiuShokurov2023optimal,
      title={Optimal bounds on surfaces}, 
      author={Jihao Liu and V. V. Shokurov},
      JOURNAL = {to appear in Algebraic Geometry and Physics},
      year={2023},
      eprint={2305.19248},
      archivePrefix={arXiv},
      primaryClass={math.AG},
      note = {\href{https://arxiv.org/abs/2305.19248}{arXiv:2305.19248}, to appear in Algebraic Geometry and Physics}
}

@book {KM98,
    AUTHOR = {Koll\'{a}r, J\'{a}nos and Mori, Shigefumi},
     TITLE = {Birational geometry of algebraic varieties},
    SERIES = {Cambridge Tracts in Mathematics},
    VOLUME = {134},
      NOTE = {With the collaboration of C. H. Clemens and A. Corti,
              Translated from the 1998 Japanese original},
 PUBLISHER = {Cambridge University Press, Cambridge},
      YEAR = {1998},
     PAGES = {viii+254},
      ISBN = {0-521-63277-3},
   MRCLASS = {14E30},
  MRNUMBER = {1658959},
MRREVIEWER = {Mark\ Gross},
       DOI = {10.1017/CBO9780511662560},
       URL = {https://doi.org/10.1017/CBO9780511662560},
}

@article {AlexeevLiu19opensurfaces,
    AUTHOR = {Alexeev, Valery and Liu, Wenfei},
     TITLE = {Open surfaces of small volume},
   JOURNAL = {Algebr. Geom.},
  FJOURNAL = {Algebraic Geometry},
    VOLUME = {6},
      YEAR = {2019},
    NUMBER = {3},
     PAGES = {312--327},
      ISSN = {2313-1691,2214-2584},
   MRCLASS = {14J29},
  MRNUMBER = {3938621},
MRREVIEWER = {Lei\ Zhang},
       DOI = {10.14231/AG-2019-015},
       URL = {https://doi.org/10.14231/AG-2019-015},
}

@article {Totaro24minimalvolume,
    AUTHOR = {Totaro, Burt},
     TITLE = {klt varieties with conjecturally minimal volume},
   JOURNAL = {Int. Math. Res. Not. IMRN},
  FJOURNAL = {International Mathematics Research Notices. IMRN},
      YEAR = {2024},
    NUMBER = {1},
     PAGES = {462--491},
      ISSN = {1073-7928,1687-0247},
   MRCLASS = {14C20 (14E30 14J45)},
  MRNUMBER = {4686657},
       DOI = {10.1093/imrn/rnad047},
       URL = {https://doi.org/10.1093/imrn/rnad047},
}

@book {CoxLittleSchenck11Toricvarieties,
    AUTHOR = {Cox, David A. and Little, John B. and Schenck, Henry K.},
     TITLE = {Toric varieties},
    SERIES = {Graduate Studies in Mathematics},
    VOLUME = {124},
 PUBLISHER = {American Mathematical Society, Providence, RI},
      YEAR = {2011},
     PAGES = {xxiv+841},
      ISBN = {978-0-8218-4819-7},
   MRCLASS = {14M25 (05A15 05E45 52B12)},
  MRNUMBER = {2810322},
MRREVIEWER = {Ivan\ Arzhantsev},
       DOI = {10.1090/gsm/124},
       URL = {https://doi.org/10.1090/gsm/124},
}

@article{alexeevliu19picardone,
  author = {Alexeev, Valery and Liu, Wenfei},
  title = {Log surfaces of {P}icard rank one from four lines in the plane},
  journal = {Eur. J. Math.},
  fjournal = {European Journal of Mathematics},
  volume = {5},
  year = {2019},
  number = {3},
  pages = {622--639},
  issn = {2199-675X,2199-6768},
  mrclass = {14J29 (14J26 14R05)},
  mrnumber = {3993254},
  mrreviewer = {Francesco\ Polizzi},
  doi = {10.1007/s40879-019-00347-2},
  url = {https://doi.org/10.1007/s40879-019-00347-2},
}

@article {Liu26positivegenus,
    AUTHOR = {Liu, Wenfei},
     TITLE = {The minimal volume of log canonical surfaces of general type
              with positive geometric genus},
   JOURNAL = {Israel J. Math.},
  FJOURNAL = {Israel Journal of Mathematics},
    VOLUME = {272},
      YEAR = {2026},
    NUMBER = {1},
     PAGES = {231--290},
      ISSN = {0021-2172,1565-8511},
   MRCLASS = {14J29},
  MRNUMBER = {5014058},
       DOI = {10.1007/s11856-025-2815-1},
       URL = {https://doi.org/10.1007/s11856-025-2815-1},
}

@book {Miyanishi94,
    AUTHOR = {Miyanishi, Masayoshi},
     TITLE = {Algebraic geometry},
    SERIES = {Translations of Mathematical Monographs},
    VOLUME = {136},
      NOTE = {Translated from the 1990 Japanese original by the author},
 PUBLISHER = {American Mathematical Society, Providence, RI},
      YEAR = {1994},
     PAGES = {xii+246},
      ISBN = {0-8218-4615-9},
   MRCLASS = {14-01 (13-01 14Exx 14Jxx)},
  MRNUMBER = {1284715},
MRREVIEWER = {Jaros\l aw\ A.\ Wi\'sniewski},
       DOI = {10.1090/mmono/136},
       URL = {https://doi.org/10.1090/mmono/136},
}

@misc{alexeevliuschutt2025modulispacestablesurfaces,
  title = {On the moduli space of stable surfaces with $p_g=1$ realizing the minimal volume},
  author = {Valery Alexeev and Wenfei Liu and Matthias Schütt},
  year = {2025},
  eprint = {2510.17678},
  archiveprefix = {arXiv},
  primaryclass = {math.AG},
  url = {https://arxiv.org/abs/2510.17678},
  note = {\href{https://arxiv.org/abs/2510.17678}{arXiv:2510.17678}}
}

\end{document}